\documentclass[11pt, a4paper]{article}

\usepackage[english]{babel}
\usepackage{enumitem}
\usepackage{amsmath}
\usepackage{amsthm}
\usepackage{amssymb}
\usepackage{tikz}
\usepackage{subcaption}
\usetikzlibrary{positioning,intersections,through,decorations.pathreplacing,calligraphy}
\usepackage{lastpage}
\usepackage{thmtools}
\usepackage{thm-restate}
\usepackage{tablefootnote}
\usepackage{footnote}
\usepackage{diagbox}
\usepackage{comment}

\newcommand{\YZ}[1]{{#1}}

\newcommand{\GG}[1]{{#1}}
\newcommand{\MA}[1]{{#1}}

\newcommand{\AY}[1]{{#1}}
\newcommand{\GGn}[1]{{#1}}
\newcommand{\YZn}[1]{{#1}}
\newcommand{\MAn}[1]{{#1}}
\newcommand{\FIX}[1]{{#1}}
\newcommand{\MAYZR}[1]{{#1}}
\newcommand{\GGR}[1]{{#1}}

\usetikzlibrary{decorations.pathmorphing}
\tikzstyle{vertexB}=[circle,draw, minimum size=18pt, scale=0.55, inner sep=0.0pt]
\tikzstyle{vertexS}=[circle,draw, minimum size=18pt, scale=0.52, inner sep=0.0pt]
\tikzstyle{arc}=[->,line width=0.03cm]

\theoremstyle{definition}

\newcommand{\2}{\vspace{0.2cm}}

\newtheorem{thm}{Theorem}
\newtheorem{lemma}{Lemma}
\newtheorem{cor}{Corollary}

\newtheorem{remarks}{Remark}
\newtheorem{conj}{Conjecture}
\newtheorem{problem}{Problem}

\usepackage[margin=3cm]{geometry}

\newcommand{\fasd}{{\rm{fasd}}}

\newcommand{\fas}{{\rm{fas}}}

\newcommand{\bas}{{\rm bas}}
\newcommand{\D}{\mathcal{D}}

\author{ Gregory Gutin\thanks{Department of Computer Science. Royal
    Holloway University of London, UK, and School of Mathematical Sciences
and LPMC, Nankai University, Tianjin 300071, China.  {\tt
g.gutin@rhul.ac.uk}}    \hspace{2mm} Mads Anker Nielsen\thanks{Department of Mathematics and
Computer Science, University of Cologne, Germany. {\tt m.nielsen@uni-koeln.de}
} \hspace{2mm}  Anders Yeo\thanks
{Department of Mathematics and Computer Science, University of Southern
Denmark, Denmark, and Department of Mathematics and Applied Mathematics,
University of Johannesburg, South Africa. {\tt yeo@imada.sdu.dk}}
\hspace{2mm}  Yacong Zhou\thanks{Shenzhen Institutes of Advanced
Technology, Chinese Academy of Sciences. {\tt yacong.zhou96@gmail.com}} }
\date{\today}
\title{Feedback Arc Sets and  Feedback Arc Set Decompositions in Weighted and Unweighted Oriented
	Graphs}

\begin{document}

\date{}

	\maketitle
	\begin{abstract}
\GGR{Let $D=(V(D),A(D))$ be a digraph with at least one directed cycle. 
A set $F$ of arcs is a feedback arc set (FAS) if $D-F$ has no directed cycle. 
The FAS decomposition number ${\rm fasd}(D)$  of $D$ is the maximum number of pairwise disjoint FASs whose union is $A(D)$. 
The directed girth $g(D)$ of $D$ is the minimum length of a directed cycle of $D$. Note that ${\rm fasd}(D)\le g(D).$
The FAS decomposition number appears in the well-known and far-from-solved conjecture of Woodall (1978) stating that for every planar digraph $D$ with at least one directed cycle,
${\rm fasd}(D)=g(D).$ The degree of a vertex of $D$ is the sum of its in-degree and out-degree. 

Let $D$ be an arc-weighted digraph and let ${\rm fas}_w(D)$ denote the minimum weight of its FAS. 
In this paper, we study bounds on ${\rm fasd}(D)$, ${\rm fas}_w(D)$ and ${\rm fas}(D)$ for 
 arc-weighted oriented graphs $D$ (i.e., digraphs without opposite arcs) with upper-bounded maximum degree $\Delta(D)$
and lower-bounded $g(D)$. Note that these parameters are related: ${\rm fas}_w(D)\le w(D)/{\rm fasd}(D)$, where $w(D)$ is the total weight of $D$, and
${\rm fas}(D)\le |A(D)|/{\rm fasd}(D).$ In particular, we prove the following: (i) If $\Delta(D)\leq~4$ and $g(D)\geq 3$, then ${\rm fasd}(D) \geq 3$ and
therefore ${\rm fas}_w(D)\leq \frac{w(D)}{3}$ which generalizes a known tight bound for an
unweighted oriented graph with maximum degree at most 4; (ii) If $\Delta(D)\leq 3$ and  $g(D)\in \{3,4,5\}$, then ${\rm fasd}(D)=g(D)$; 
(iii)  If $\Delta(D)\leq 3$ and  $g(D)\ge 8$ then ${\rm fasd}(D)<g(D).$
We also give some bounds for the cases when $\Delta$ or $g$ are large and state several open problems and a conjecture.
}

\end{abstract}

	\section{Introduction}
	   For terminology and notation on digraphs not introduced in this paper, see \cite{BJG}.
\GGR{A digraph $D=(V(D),A(D))$ is {\em weighted} if $D$ is given together with a function $w:\ A(D) \to \mathbb{R}_{\ge 0}.$ 
}	
	A set $F\subset A(D)$ of a digraph $D$ is a {\em feedback arc set \GGR{(FAS)}} if
    $D-F$ is acyclic. In the {\em unweighted (weighted, respectively)
    feedback arc set problem}, given an unweighted (weighted, respectively)
    digraph $D$, we are to find an FAS $F$ of minimum size
    (weight, respectively), denoted by $\fas(D)$ {($\fas_w(D)$, respectively)}. The problem is {\sf
    NP}-hard even on unweighted tournaments \cite{Alon2006,CharbitTY2007}
    and \MAYZR{both unweighted and weighted
    	versions of it have numerous applications}, see e.g. \cite{Alon2002,ELS1993,FLRS,LS1991}. As the problem
    is of great theoretical and practical interest, its various aspects
    have been studied including approximation, exact, heuristic, and
    parameterized algorithms, computational complexity, and upper and lower bounds. 
    
    Lower and upper bounds are \GGR{usually} studied for {\em oriented graphs} (or, {\em
    orgraphs}) i.e. digraphs without directed 2-cycles because \GGR{we can reduce the FAS problem on digraphs to that on orgraphs as follows: 
    If a digraph $D$ is unweighted, we can delete all arcs in directed 2-cycles reducing the size of minimum FAS by the number of directed 2-cycles.
    Indeed, exactly one of the two arcs in a directed 2-cycle \MAYZR{is in any
    minimal} FAS. If $D$ is weighted, for every directed 2-cycle $C$, we reduce the weights of 
    both arcs of $C$ by subtracting from each weight the minimum weight of the two. Then the arcs with zero weight can be deleted.  
} 


    While almost all research on upper and lower bounds \GGn{for the problem} has been on
    unweighted \MAYZR{orgraphs},  in this paper, we study upper and lower bounds
    for weighted orgraphs. We take into consideration not only arc weights,
    but also the maximum degree (the {\em degree} of a vertex is the sum of its
    in- and out-degrees) and directed girth (the minimum  number of arcs in
    a directed cycle or $\infty$, if there are no directed cycles) of
    \MAYZR{orgraphs}.  The maximum degree already appeared in the the following
    well-known upper  bounds of Berger and Shor \cite{BS1990,BS1997}:
    ${\fas}(D)\le (\frac{1}{2}-\Omega(\frac{1}{\sqrt{\Delta}}))a(D),$ where
    $\Delta$ \AY{is} the maximum degree of an oriented graph $D$ and $a(D)$ the
    number of arcs in $D$, and of Alon \cite{Alon2002}: $\YZn{\fas_w(D)}\le
    (\frac{1}{2}-\frac{1}{16\sqrt{2\Delta}})w(D),$ where $w(D)$ is the sum
    of weights of a weighted orgraph $D$, and $\Delta$ is the
    maximum degree of $D$. Note that by Jung \cite{Jung1970} and Spencer
    \cite{Spencer1971}, the bounds of Berger and Shor, and Alon are tight
    subject to a coefficient $b$ in $b/\sqrt{\Delta}$. Alon and Seymour,
    see \cite[Section 3]{Seymour95} observed that there are 3-regular
    orgraphs with $n$ vertices, $m$ arcs and directed girth at least
    $\frac{4}{5}\ln n,$ where every feedback arc set is of size at least
    $m/24$.  
	
    Let $\D_{\Delta,g}$ be the set of weighted  orgraphs of maximum
    degree at most $\Delta$ and directed girth at least $g.$ Let
    $\fas_w(\Delta,g)$ denote the supremum of the set $\{\YZn{\fas_w(D)}/w(D): \
    D\in \D_{\Delta,g}\}$. The same parameter restricted to unweighted
    orgraphs will be denoted by $\fas(\Delta,g)$. By \cite{Alon2002} and
    \cite{Seymour95}, we have $\fas(\Delta,3)\le \fas_w(\Delta,3)\le
    \frac{1}{2}-\frac{1}{16\sqrt{2\Delta}}$ and $\fas(6,\lceil 4\ln
    n/5\rceil)\ge \frac{1}{24}.$ Upper bounds and exact values of
    $\fas(\Delta,3)$ for small values of $\Delta$ have been studied in
    \cite{BS1990,BS1997,EL1995,GLYZ,Hanauer2017,HBA2013}. In particular,
    Berger and Shor \cite{BS1997} proved that \GGn{$\fas(3,3) = \frac{1}{3}$},
    $\fas(4,3)\le \frac{11}{30}$ and $\fas(5,3)\le \frac{11}{30}$. 
    Hanauer, Brandenburg and Auer  \cite{HBA2013}  proved that
    $\fas(4,3)=\frac{1}{3}$, and Gutin, Lei, Yeo, and Zhou \cite{GLYZ} showed that $\fas(5,3)=\frac{1}{3}.$ 
	
    \GGR{Usually, upper bounds for
    $\fas(\Delta,g)$ cannot be easily extended to $\fas_w(\Delta,g)$, if such an extension is possible. 
    However,  $\fas(\Delta,g)\le a(D)/\fasd(D)$
    can be immediately extended to $\fas(\Delta,g)\le w(D)/\fasd(D),$ where $\fasd(D)$ is the maximum natural number $t$ such that $A(D)$ can
    be partitioned into $t$ FASs of $D$. Note that  $\fasd(D)$ equals the maximum number of pairwise-arc-disjoint FASs 
  of $D$, and $\fasd(D)\le g(D)$ as for a decomposition $\cal F$ of $A(D)$ into $\fasd(D)$ feedback arc sets and a directed cycle $C$ of length $g(D)$, every FAS in $\cal F$ must include an arc from $C$.
  
  The parameter $\fasd(D)$  was already used by Woodall \cite{Woo1978} who
  conjectured that for every non-acyclic planar digraph $D$,
  $\fasd(D)=g(D).$ This conjecture remains open and is confirmed only for
  special classes of digraphs (which include non-planar digraphs as their
  members): Lee and Wakabayashi \cite{LW2001} proved that the conjecture
  holds for digraphs whose underlying graph is of treewidth at most 2; Lee
  and Williams \cite{LW2006} confirmed the conjecture for planar digraphs
  with no $K_5 - e$ minor;  Guti{\'e}rrez \cite{Gutier} verified the
  conjecture for arc-maximal digraphs whose underlying graph is of
  treewidth 3 and for digraphs of treewidth at most 3 and directed girth 3.
  Our results imply new classes of orgraphs where $\fasd(D)=g(D)$ always
  holds and other classes where it is not the case. 
  }
    
    Let $\sigma =
    v_1v_2\dots v_n$ be an ordering of $V(D)$. For $v_iv_j \in A(D)$ we say
    that $v_iv_j$ is a \textit{backward arc} with respect to $\sigma$ if $j
    < i$ and otherwise it is a \textit{forward arc}. Note that for every
    ordering of $V(D)$ the backward arcs form a feedback arc set and so
    \MA{do the} forward arcs. Hence, $\fasd(D)\ge 2$ for every non-acyclic orgraph $D$. 
    For an acyclic orgraph $D$, we set $\fasd(D)=\infty$.
	
	Let $\fasd(\Delta,g)$ be the minimum of $\fasd(D)$ over all orgraphs $D$ with maximum degree at most $\Delta$ and directed girth at least $g.$ Then, \MAYZR{clearly $2\le\fasd(\Delta,g)\le g.$} 
	 \MAYZR{In addition,}
	\begin{equation}\label{main-bound}
	\fasd(\Delta,g)\le \lfloor 1/\fas(\Delta,g) \rfloor,
	\end{equation} 
	and for every arc-weighted $D\in \D_{\Delta,g}$, 
	\begin{equation}\label{main-bound1}
	\YZn{\fas_w(D)}\le w(D)/\fasd(\Delta,g).
	\end{equation} 


In Section \ref{sec: fas_w(4,3)}, we prove that $\fas_w(4,3)=\frac{1}{3}$,
which generalizes the result $\fas(4,3)=\frac{1}{3}$ by Hanauer,
Brandenburg, and Auer  \cite{HBA2013}. In fact, we prove a stronger result:
$\fasd(4,3)=3.$ In Section \ref{sec: fasd(3,4)=4}, we prove that
$\fasd(3,g)=g$ for $g\in \{3,4,5\}$ implying that $\fas_w(3,g)\le
\frac{1}{g}$ for $g\in \{3,4,5\}$. \MAYZR{While we are not able to extend
this proof to $g=6$, we do show that $\fas(3,6) = \frac{1}{6}$. This
provides some support for the possibilty that $\fas_w(3,6) =
\frac{1}{6}$ and $\fasd(3,6) = 6$.}

\AY{In Section~\ref{sec:lub} we show that $\fasd(3,g)<g$ when \GGR{$g\ge 8$}, so it would be interesting to determine for which $g$ we have $\fasd(3,g)=g$ and $\fasd(3,g+1)<g+1$. Clearly 
$g \in \{5,6,7\}$ in this case.} 
	Note that for $\Delta=2$ we have $\fas(2,g)=\frac{1}{g}$ and
    $\fasd(2,g)=g$ and therefore $\fas(2,g)\to 0$ and $\fasd(2,g)\to
    \infty$ as $g\to \infty$. Intuitively, this trend should still be true
    if we fix a larger $\Delta$. However, we show that this is not the
    case. In particular, in Section \ref{sec:lub}, we prove the following.
    For any integer $g\geq 3$ and prime power $p \MAYZR{\ \equiv 1 \pmod 4}$, there exists a $\frac{p+1}{2}$-regular orgraph $D$ with directed girth at least $g$ such that $\fas(D)\geq \frac{p+1-2\sqrt{p}}{4(p+1)}a(D)$ and therefore $\fasd(D)\leq \frac{4(p+1)}{p+1-2\sqrt{p}}$. Thus, $\fasd(p+1,g)\leq \frac{4(p+1)}{p+1-2\sqrt{p}}$ for every $g\ge 3.$ Using this result and a vertex splitting operation, we show that 
if $\Delta\geq 3$, then for every $g\geq 3$ we have $\fas(\Delta,
g)>\frac{1}{\FIX{95}}$ and $\fasd(\Delta,g)\le \FIX{94}$. \YZ{We also show that there is an integer $0<c\leq \MAYZR{1362}$ such that if $\Delta\geq c$ then $\fas(\Delta,g)>1/3$ and therefore by (\ref{main-bound}), $\fasd(\Delta,g)=2$. In other words, letting $g$ be arbitrarily large does not help in getting a better lower bound than just the trivial $\fasd(\Delta,g)\geq 2$, if we do not bound the $\Delta$ by a relatively small constant and especially by \MAYZR{1362}.} In Section  \ref{sec:lub}, we also prove that $\fasd(5,4) \le 3$ and $\fasd(4,6)\le 5$.	

\YZn{In Section \ref{sec:conclusion}, we  conclude our paper by stating some open problems, including the above one. }

In \MAYZR{Table}~\ref{table:Results} we summarize the results for
$\fasd(\Delta,g)$ obtained in this paper \YZn{where bounds on some entries
are obtained from the fact that \[\fasd(\Delta+1,g)\leq \fasd(\Delta,g)\leq
\fasd(\Delta,g+1),\] which follows directly from the definition of
$\fasd(\cdot,\cdot)$. For example, for every $g\geq 5$, we have that
$\fasd(\Delta,g)\leq \fasd(102,g)\leq 4$ when $\Delta> 102$. And for
$\Delta=3$, we have that $\fasd(3, g)\geq \fasd(3,5)=5$ when $g>5$.}
\MAYZR{Note
that for entries containing $< g$ or an open interval excluding $g$, the
exact upper bound comes from Theorem \ref{lemB} and is stronger than just $<g$
for larger $g$.}
\FIX{Note that one can obtain a better upper bound than 15 for $\fasd(\Delta,g)$
for those $\Delta \in [7,101]$ which are one more than an odd prime power
$p \equiv 1 \pmod 4$
using Theorem \ref{thm:girthandlargefas} in Section \ref{sec:lub}. For
example, one can show that $\fasd(\Delta,g) \leq 8$ when $\Delta \geq 14$
and $\fasd(\Delta, g) \leq 5$ when $\Delta \geq 38$.}

\begin{table}[h]
\begin{center}
	\scalebox{.9}{\begin{tabular}{|c||c|c|c|c|c|c|c|c|c|c|c|} \hline
\MAYZR{\diagbox[width=2.5cm,height=1cm]{$~~~~~g$}{$\Delta~~$}}       
&    $2$    &    3      &    4      &    5     &     6      & $\cdots$ &  102   & $\cdots$&  \MAYZR{422}  & $\cdots$ & $\geq$ \MAYZR{1362}     \\ \hline \hline
$g=3$                &    $3$    &    $3$    &    $3$    &          &           \YZn{$2$}\footnotemark&      2    &     2     &    2    &  2&  2&         2      \\ \hline
$g=4$                &    $4$    &    $4$    &  $\in \{3,4\}$       & $\leq 3$ &     $\leq 3$       &     $\leq 3$     &     $\leq 3$     &  $\leq 3$    & $\leq 3$   & $\leq 3$ &  2             \\ \hline
$g=5$                &    $5$    &    $5$    &     $\in[3,5]$      &          &            &          & $\leq 4$ &   $\leq 4$   &  $\leq 3$     & $\leq 3$ &  2             \\ \hline
$g=6$              &    $6$    &       $\in \{5,6\}$      & $\in \MAYZR{[3,5]}$  &    $\leq 5$       &      $\leq 5$       &      $\leq 5$     & $\leq 4$ &  $\leq 4$    &  $\leq 3$      & $\leq 3$   & 2             \\ \hline
$g=7$             &    $\MAYZR{7}$    &     $\in [5,7]$     &     $\in [3,\MAYZR{7}]$      &          &            &          & $\leq 4$ &  $\leq 4$     & $\leq 3$    & $\leq 3$  & 2            \\ \hline
\MAYZR{$8 \leq g \leq 15$}  &    $g$    &     $\in [5,\MAYZR{g)}$      &    $\in
[3,\MAYZR{g)}$        &    \MAYZR{$<g$}      &     \MAYZR{$< g$} &       \MAYZR{$ < g$}   & $\leq 4$ &    $\leq 4$   & $\leq 3$    &  $\leq 3$  &2             \\ \hline
\MAYZR{$16 \leq g \leq 94$}  &    $g$    &    $\in [5,\MAYZR{g)}$        & $\in [3,\MAYZR{g)}$        &     \MAYZR{$ < g$}    & $\leq 15$  &    $\leq 15$      & $\leq 4$ & $\leq 4$   & $\leq 3$    &  $\leq 3$ & 2             \\ \hline
$g \geq \FIX{95}$          &    $g$    & $\MAYZR{\in [5,94]}$ & $\MAYZR{\in [3,94]}$ & $\FIX{\leq 94}$      & $\leq 15$  &     $\leq 15$     & $\leq 4$ &  $\leq 4$  &  $\leq 3$       & $\leq 3$ & 2             \\ \hline
\end{tabular}}
\end{center}

\caption{Our main results for the value of $\fasd(\Delta,g)$. Where no
entry is listed, we only know that $2 \leq \fasd(\Delta,g) \leq g$. \GGR{Using Theorem \ref{lemB} the upper bound for $\fasd(3,g)$ can be  improved for every even $12\le g\le 122$ and every odd $15\le g\le 121$.}}
\label{table:Results}
\end{table}
\footnotetext[1]{\YZn{This value is implied by a result in \cite[Proposition 11]{GLYZ}}.}

\vspace{2mm}

We conclude this section with some {\bf additional terminology and notation}.
Let $D$ be a digraph and let $v \in V(D)$. The out-degree
(in-degree, respectively) is denoted $d^+_D(v)$ ($d^-_D(v)$, respectively). Recall that the {\em degree} of $v$ is
$d_D(v) = d^+_D(v) + d^-_D(v)$. The maximum degree $\Delta(D)$ of $D$ is
defined as $\Delta(D) = \max_{v \in V(D)} d_D(v)$. 
A digraph $D$ is {\em
$k$-regular} if $d^+_D(v) =d^-_D(v)=k$ for every vertex $v\in V(D)$. For a
positive integer $k$, \MA{we define} $[k]=\{1,2,\dots ,k\}.$

The {\em order} of a directed or undirected graph $H$ is the number of
vertices in $H.$	In a digraph, a {\em cycle} ({\em path}, respectively)
is a directed cycle (directed path, respectively). 	A {\em $k$-cycle} is a
cycle with $k$ vertices. The {\em underlying graph} of a digraph $D$, is
the undirected graph $U(D)$  with the same vertex set as $D$ and such that
a pair $u,v$ of distinct vertices are adjacent in $U(D)$ if there is an arc
between $u$ and $v$ in $D.$ A component of $U(D)$ is a {\em component} of
$D$ and $D$ is {\em connected} if $U(D)$ is connected.
	
\MAYZR{
For a positive integer $g$, a \emph{$g$-arc-coloring} of an orgraph $D$ is a map
\(
c : A(D) \to [g].
\)
Such a coloring is called \emph{good} if every directed cycle of $D$ contains all $g$ colors.}

	\section{Proving $\fasd(4,3)=3$}\label{sec: fas_w(4,3)}
	
    \MA{
    Let $D \in \D_{4,3}$. The proof relies on constructing a triple
    $(\sigma_1,\sigma_2,\sigma_3)$ of orderings of $V(D)$. We call such a
    triple \textit{good} if for every arc $a \in A(D)$, $a$ is a backward
    arc with respect to $\sigma_i$ for exactly one $i \in [3]$. If $F$ is
    the set of backward arcs with respect to an ordering $\sigma$, then
    $D-F$ is acylic and thus $F$ is a feedback arc set of $D$. Hence, a
    good triple $T = (\sigma_1,\sigma_2,\sigma_3)$ of $D$ induces a
    partition of $A(D)$ into three feedback arc sets $F_1$, $F_2$, and
    $F_3$ where $F_i$ is the set of backward arcs with respect to
    $\sigma_i$ for $i=1,2,3$. Our goal is therefore to construct a good
    triple of $D$.}

    If $(\sigma_1,\sigma_2,\sigma_3)$ is a good triple of $D$ and there
    exists a vertex $v \in V(D)$ such that $v$ is first in $\sigma_1$ and
    last in $\sigma_2$, then we say that $(\sigma_1,\sigma_2,\sigma_3)$ is
    a {\em $v$-triple}. As a notational convention, we write $\sigma^v$ to
    emphasize that $v$ is the first vertex of the ordering $\sigma^v$.
    Likewise, we write $\sigma_v$ to emphasize that $v$ is the last vertex
    of the ordering $\sigma_v$. 

    Let $D$ be an oriented graph  and let $T'$ be a good triple of $D-v$
    for some $v \in V(D)$. By \textit{inserting} $v$ into $T'$ we mean
    inserting $v$ into every ordering of \MA{$T'$} in such a
    way that we obtain a good triple $T$ of $D$. We say that $\sigma'$ is a
    {\em subordering} of $\sigma$ if $\sigma'$ can be obtained from
    $\sigma$ by deleting vertices. We write $\sigma' \leq \sigma$ if
    $\sigma'$ is a subordering of $\sigma$.

	The main theorem of this section is the following:
	
	\begin{restatable}{thm}{main}
		\label{thm:main}
		For any $H \in \D_{4,3}$, $A(H)$ can be partitioned into 3 feedback arc
		sets.
	\end{restatable}
	
	Since $\fas(4,3)=\frac{1}{3}$, the following corollary holds.
	
	\begin{cor}
		We have $\fasd(4,3)=3$ and $\fas_w(D)\le w(D)/3$ for every arc-weighted $D\in \D_{4,3}$. 
	\end{cor}
	
    We prove three lemmas in the first subsection of this section and the
    main theorem in the second subsection. \MAYZR{For use in the following
    two subsections, we define} an \textit{unbalanced} vertex $v \in V(D)$
    \MAYZR{as} a vertex such that $\min\{d^+_D(v),d^-_D(v)\} \leq 1$. Furthermore,
    the \textit{converse} of a digraph $D$ is the digraph $D'$ obtained by
    reversing all arcs of $D$.
	
	\subsection{\MAYZR{Preliminary Results}}
    \MAYZR{We start by making the following straightforward observations.}
	
	\begin{restatable}{observation}{obs:one}
		\label{obs:one}
		Let $D$ be a digraph, let $\MAYZR{x} \in V(D)$, and let $\sigma'$ be an ordering
		of $V(D)-x$. If we obtain $\sigma$ by inserting $x$ into $\sigma'$ such
		that all in-neighbors of $x$ lie before $x$ and all out-neighbors of
		$x$ lie after $x$, then there are no backward arcs incident with $x$
		with respect to $\sigma$. 
	\end{restatable}
	
	We use Observation \ref{obs:one} in combination with the following
	observation.
	
	\begin{restatable}{observation}{obs}
		\label{obs}
		Let $D$ be a digraph, let $x \in V(D)$ and let $T =
		(\sigma_1,\sigma_2,\sigma_3)$ be a good triple of $D-x$. If we obtain
		$\sigma$ by inserting $x$ into $\sigma_1$ such that \YZ{there are no backward arcs incident with $x$} with respect $\sigma$, then
		$(x\sigma_2,\sigma_3x,\sigma)$ and $(x\sigma_3,\sigma_2 x,\sigma)$ are
		good ($x$-)triples of $D$.
	\end{restatable}

    \MAYZR{
        \begin{proof}[Proof of Observation~\ref{obs}]
            We show that $(x\sigma_2,\sigma_3x,\sigma)$ is good. The
            proof for $(x\sigma_3,\sigma_2x, \sigma)$ is identical. We only
            need to show that any arc incident with $x$ is backward with
            respect to exactly one of $x\sigma_2$, $\sigma_3x$, and $\sigma$
            as the triple $T$ that we started with is good. Clearly, exactly
            the arcs entering $x$ are backward with respect to $x\sigma_2$
            ($x$ is first in this ordering) and exactly the arcs leaving
            $x$ are backward with respect to $\sigma_3x$. Furthermore,
            $\sigma$ was obtained by inserting $x$ into $\sigma_1$ such
            that no arcs incident with $x$ are backward.
        \end{proof}
    }
	
	We insert vertices into triples several times in our proofs. In all cases,
	the argument that the resulting triple is good is a combination of
	Observation \ref{obs:one} and \ref{obs}.

    \MAYZR{We now prove the following lemma, which} gives a proof of the main
    theorem in the case where $H$ contains no $2$-regular component.
	
	\begin{restatable}{lemma}{nonregular}
		\label{lem:nonregular}
		Let $D \in \D_{4,3}$ be such that no component of $D$ is $2$-regular. For
		all unbalanced $v \in V(D)$, there exists a good $v$-triple
		$(\sigma^v,\sigma_v,\sigma)$ of $D$.
	\end{restatable}
	
	\begin{proof}
        By induction on the order of $D$. If $V(D) = \{v\}$, then $(v,v,v)$
        is a good $v$-triple.
		
        Suppose $|V(D)| \geq 2$ and let $v \in V(D)$ be an unbalanced
        vertex. We may assume without loss of generality that $d^-(v) \leq
        1$ since we can otherwise consider the converse of $D$. Let $D' =
        D-v$. Observe that we may apply the induction hypothesis to $D'$
        since any component of $D'$ is either a component of $D$ or
        incident with $v$ in $D$ and thus contains a vertex of degree at most
        $3$. We consider two cases.
		
		\quad
		
		\noindent\textbf{Case 1}: $d^-(v) = 0$.
		
        We have $|V(D')| \geq 1$ and since $D'$ is not $2$-regular, $D'$
        contains at least one unbalanced vertex. By the induction
        hypothesis, there exists a good triple
        $(\sigma_1,\sigma_2,\sigma_3)$ of $D'$. We claim that
        $(v\sigma_1,\sigma_2v,v\sigma_3)$ is a good $v$-triple of $D$.
        Indeed, $v$ has no in-neighbors, so \YZ{there are no backward arcs
        incident with $v$} with respect to $v\sigma_1$ or $v\sigma_3$.
        Furthermore, all arcs incident with $v$ are \YZ{backward arcs} with
        respect to $\sigma_2v$. This completes the proof of Case 1.
		
		\quad
		
		\noindent\textbf{Case 2}: $d^-(v) = 1$.
		
        The vertex $u \in V(D')$ such that $N^-_D(v) = \{u\}$ is unbalanced
        since $d_{D'}(u) \leq 3$. By the induction hypothesis, there exists
        a good $u$-triple $(\sigma^u,\sigma_u,\sigma)$ of $D'$. Let
        $\sigma^{uv}$ be the ordering obtained from $\sigma^{u}$ by
        inserting $v$ immediately after $u$. \MA{In $\sigma^{uv}$, all
        in-neighbors of $v$ (just $u$) lie before $v$ and all other
        vertices lie after $v$. In particular, all out-neighbors of $v$ lie
        after $v$ ($u$ is not an out-neighbor of $v$ as $D$ is oriented).} Thus,
        \YZ{there are no backward arcs incident with $v$} with respect to
        $\sigma^{uv}$. Now, $(v\sigma_u,\sigma v,\sigma^{uv})$ is a good
        $v$-triple of $D$, where the arcs leaving $v$ are backward
        \YZ{arcs} with respect to $\sigma v$ and the arcs entering
        $v$ (i.e. $uv$) are backward \YZ{arcs} with respect to $v
        \sigma_u$. This completes the proof of Case 2.
		
		Either Case 1 or Case 2 applies, and thus we have proven the lemma.
	\end{proof}

    The {\em transitive triangle} is the orgraph
    $(\{a,b,c\},\{ab,ac,bc\})$. We now show the following lemma, which
    gives a proof of the main theorem in the case where $H$ is connected,
    $2$-regular, and contains a transitive triangle.
	
	\begin{restatable}{lemma}{transitive}
		\label{lem:transitive}
		Let $D \in \D_{4,3}$ be $2$-regular and connected. If $D$ contains a
		transitive triangle, then there exists a good triple of $D$.
	\end{restatable}
	
	\begin{proof}
		Let $a_1,a_2,x \in V(D)$ be such that $\{a_1a_2,a_1x,a_2x\} \subseteq
		A(D)$. Let $D' = D-\{a_1,x\}$. Any component of $D'$ is incident with
		$\{a_1,x\}$ and thus not $2$-regular. Furthermore, $d_{D'}(a_2) = 2$
		and $a_2$ is thus unbalanced in $D'$. By Lemma \ref{lem:nonregular}, there
        exists a good $a_2$-triple $T' = (\sigma^{a_2},
        \sigma_{a_2},\sigma)$ of $D'$. Let $T = (a_1\sigma^{a_2},\sigma
        a_1, \sigma_{a_1a_2})$ where $\sigma_{a_1a_2}$ is the ordering
        obtained from $\sigma_{a_2}$ by inserting $a_1$ immediately before
        $a_2$. 

        We claim that $T$ is a good triple of $D-x$. In $\sigma_{a_1a_2}$,
        all out-neighbors of $a_1$ \MA{(just $a_2$)} are
        after $a_1$ and all in-neighbors are before $a_1$ (since only $a_2$
        is after $a_1$). Thus, \YZ{there are no backward arcs incident with
        $a_1$} with respect to $\sigma_{a_1a_2}$ (Observation
        \ref{obs:one}). Exactly the arcs leaving $a_1$ are backward arcs
        with respect to $\sigma a_1$ and exactly the arcs entering $a_1$
        are backward arcs with respect to $a_1\sigma^{a_2}$. Thus, $T$ is a
        good triple of $D-x$ (Observation \ref{obs}).
		
		In the ordering $a_1 \sigma^{a_2}$, $a_1$ and $a_2$ are the two first
		vertices. Obtain $\pi$ from $a_1\sigma^{a_2}$ by inserting $x$
		immediately after $a_1$ and $a_2$. In $\pi$, all in-neighbors of $x$
		($a_1$ and $a_2$) are before $x$ and all out-neighbors of $x$ are after
		$x$. Hence, there are no backward arcs incident with $x$ with respect to
        $\pi$. \MA{Now, $(x\sigma a_1,\sigma_{a_1a_2}x,\pi)$ is a good triple
        of $D$, which completes the proof.}
	\end{proof}

    \MA{The last lemma we need is the following technical lemma. An
        \textit{anti-directed} path in a digraph $D$ is a path $P =
        x_1x_2\dots x_\ell$ in the underlying graph $U(D)$ of $D$ such that
        $x_{i-1}x_{i} \in A(P)$ implies $x_{i+1}x_{i} \in A(P)$ and $x_ix_{i-1} \in A(P)$ implies $x_i x_{i+1} \in A(P)$
        for $i=2,3,\dots,\ell-1$ (i.e. the direction of the arcs on $P$ alternates).}
	
	\begin{restatable}{lemma}{antidirected}
		\label{lem:antidirected} 
        Let $D \subseteq H$ for some $2$-regular \MA{connected} $H \in \D_{4,3}$, let $P =
		x_1x_2\dots x_\ell$ be an anti-directed path in $D$ of with $V(P) \geq
		3$ such that $N_D^+(x_1) = \{x_2\}$ or $N_D^-(x_1) = \{x_2\}$, and let
		$T = (\pi^{x_\ell},\pi_{x_\ell},\pi)$ be any good
		$x_\ell$-triple of $D-(V(P) \setminus \{x_\ell \})$. For any $\pi^* \in
		\{\pi^{x_\ell},\pi_{x_\ell}\}$ there exists 
		\begin{enumerate}[label=(\arabic*)]
			\item a good $x_1$-triple $T_1 =
			(\sigma^{x_1},\sigma_{x_1},\sigma)$ of $D$ such that $\pi^*
			\leq \sigma^{x_1}$, and
			\item a good $x_1$-triple $T_2 =
			(\sigma^{x_1},\sigma_{x_1},\sigma)$ of $D$ such that $\pi^*
			\leq \sigma_{x_1}$.
		\end{enumerate}
	\end{restatable}
	
	\begin{proof}
        The proof is by induction on the order of $P$. First, we observe
        that it suffices to consider the case where $N_D^+(x_1) = \{x_2\}$.
        Indeed, if $N_D^-(x_1) = \{x_2\}$, then consider the converse $D^R$
        of $D$, the converse anti-directed path $P^R$ of $P$, \MA{and the good
        $x_\ell$-triple $T^R = ((\pi_{x_\ell})^R, (\pi^{x_\ell})^R,\pi^R)$
        where $\sigma^R$ denotes the reverse of the order $\sigma$. Let
        $\pi^* \in \{\pi^{x_\ell},\pi_{x_\ell}\}$. If the lemma holds for
        $D^R$, $P^R$, and $T^R$, then we may obtain a good $x_1$-triple
        $T^R_2 = (\sigma^{x_1},\sigma_{x_1},\sigma)$ of $D^R$ such that
        $(\pi^*)^R \leq \sigma_{x_1}$. Now we have $\pi^* \leq
        (\sigma_{x_1})^R$ and thus
        $((\sigma_{x_1})^R,(\sigma^{x_1})^R,\sigma^R)$ is the desired good
        $x_1$-triple $T_1$ of $D$. The triple $T_2$ is obtained similarly.}

        Now, suppose $|V(P)| = 3$. Let $D' = D-(V(P) \setminus \{x_3\})$
        and let $T = (\pi^{x_3},\pi_{x_3},\pi)$ be any good $x_3$-triple of
        $D'$. Pick $\pi^* \in \{\pi^{x_3},\pi_{x_3}\}$ arbitrarily. Since
        $P$ is an anti-directed path and $N_D^+(x_1) = \{x_2\}$ we have
        $A(P) = \{x_1x_2,x_3x_2\}$.  
		
		\newcommand{\xx}{\pi^{x_3x_2}}
		
		Denote by $\xx$ the ordering obtained by inserting $x_2$ into
		$\pi^{x_3}$ immediately after $x_3$. Now let
		\begin{alignat*}{5}
			&T' &&= (&&x_1\xx\ &&,\ x_2\pi_{x_3}x_1\ &&,\ \pi x_1x_2), \\
			&T'' &&= (&&x_1x_2\pi_{x_3}\ &&,\ \xx x_1\ &&,\ \pi x_1x_2).
		\end{alignat*}
		
		If $\pi^* = \pi^{x_3}$, then let $T_1 = T'$ and let $T_2 =
		T''$. We indeed have $\pi^{x_3} \leq \xx \leq x_1\xx$ and $\pi^{x_3}
		\leq \xx \leq \xx x_1$. If $\pi^* = \pi_{x_3}$, then let $T_1 = T''$
		and let $T_2 = T'$. We indeed have $\pi_{x_3} \leq x_1x_2\pi_{x_3}$ and
		$\pi_{x_3} \leq x_2\pi_{x_3}x_1$ as desired. It remains to be shown that
		$T'$ and $T''$ are good triples.

        \2
		
        \textbf{Claim~A:}  $T'$ and $T''$ are good triples of $D$.

		\begin{proof}[Proof of Claim A]
			\renewcommand{\qedsymbol}{$\diamond$}
			
            We refer to the arcs incident with $x_1$ and $x_2$ as
            \textit{new} arcs. For each triple, we show that each new arc
            is \YZ{a backward arc} with respect to exactly one ordering in
            that triple. First, observe that \YZ{there are no backward arcs
            incident with $x_2$} with respect to $\xx$ since \MA{all
            in-neighbors of $x_2$ ($x_3$) lie before $x_2$ and all other
            vertices lie after.} For a vertex $x$, let $A^+(x)$ ($A^-(x)$) be
            the arcs leaving (entering) $x$ in $D$.

            Consider $T'$. The set of new backward arcs with respect to
            $x_1\xx$ is exactly $A^-(x_1)$. \MA{The set of new backward
            arcs} with respect to $x_2\pi_{x_3}x_1$ is exactly $A^-(x_2)
            \cup A^+(x_1)$. \MA{Lastly, the set of new backwards arcs with
            respect to} $\pi x_1 x_2$ is exactly $A^+(x_2)$ since
            $N_D^+(x_1) = \{x_2\}$. We notice that every new arc is \YZ{a
            backward arc} in exactly one ordering of $T'$, and thus
            $T'$ is good.
			
            A similar observation shows that $T''$ is good: The set of new
            backward arcs is exactly $A^-(x_1) \cup A^-(x_2)$ with respect
            to $x_1 x_2 \pi_{x_3}$,  $A^+(x_1)$ with respect to $\xx x_1$,
            and $A^+(x_2)$ with respect to $\pi x_1 x_2$.
			
			This completes the proof of the claim.
		\end{proof}

        Now, suppose $|V(P)| \geq 4$ and the lemma holds for all shorter
        $P$. Let $T = (\pi^{x_\ell},\pi_{x_\ell},\pi)$ be any good
        $x_\ell$-triple of $D - (V(P) \setminus \{x_\ell \})$. It suffices
        to consider the case where $N^+_D(x_1) = \{x_2\}$ by the comment
        made in the beginning of the proof.
		
		Consider $D' = D-x_1$ and $P' = P-x_1 = x_2x_3\dots x_\ell$. Then $P'$
		is an anti-directed path in $D'$ with fewer vertices and $V(P') \geq
		3$. Furthermore, since $P$ is an anti-directed path in $D$ and $D$ is a
		subdigraph of the $2$-regular $H$, we have $N^-_D(x_2) = \{x_1,x_3\}$
		and thus $N^-_{D'}(x_2) = \{x_3\}$. Hence, we may apply the induction
		hypothesis to $D'$, $P'$ and $T$.
		
		Let $\pi^* \in \{\pi^{x_\ell},\pi_{x_\ell}\}$. By the induction
		hypothesis, there exists a good $x_2$-triple
		$(\sigma^{x_2},\sigma_{x_2},\sigma)$ of $D'$ such that $\pi^* \leq
		\sigma^{x_2}$. Let $\sigma_{x_1x_2}$ be the ordering obtained from
		$\sigma_{x_2}$ by inserting $x_1$ immediately before $x_2$. Let 
		\[
		T_1 = (x_1\sigma^{x_2},\sigma x_1, \sigma_{x_1x_2}) \text{ and } 
		T_2 = (x_1\sigma,\sigma^{x_2} x_1, \sigma_{x_1x_2}).
		\]
		
		Both $T_1$ and $T_2$ are $x_1$-triples, and we indeed have $\pi^* \leq
		\sigma^{x_2} \leq x_1\sigma^{x_2}$ and $\pi^* \leq \sigma^{x_2} \leq
		\sigma^{x_2} x_1$. Furthermore, there are no backward arcs incident with
		$x_1$ in $\sigma_{x_1x_2}$, the backward arcs incident with $x_1$ with
		respect to $x_1\sigma$ are exactly the arcs entering $x_1$, and the
		backward arcs incident with $x_1$ with respect to $\sigma^{x_2}x_1$ are
		exactly the arcs leaving $x_1$. This completes the inductive step and
		thus the proof.
	\end{proof}
	
	\subsection{Main theorem}

    We start by explaining how we plan to obtain a good triple of an
    arbitrary $H \in \D_{4,3}$. We may assume that $H$ is connected since
    the arcs of each connected component can be partitioned separately. By
    Lemma \ref{lem:nonregular} and Lemma \ref{lem:transitive}, we may also
    assume that $H$ is $2$-regular and free of transitive
    triangles. Let $x \in V(H)$ be arbitrary and consider $D = H-x$. Let
    $N^+_H(x) = \{b_1,b_2\}$ and $N^-_H(x) = \{a_1,a_2\}$. Our goal is to
    find a good triple $T = (\sigma_1,\sigma_2,\sigma_3)$ of $V(D)$ such
    $a_1$ and $a_2$ lie before $b_1$ and $b_2$ in some \MA{ordering in}
    $T$, say $\sigma_1$. If we can find such a triple, then we can insert
    $x$ between $\{a_1,a_2\}$ and $\{b_1,b_2\}$ in $\sigma_1$ without
    introducing any backward arcs. We insert $x$ \MAYZR{as the first element in} one of the other
    orderings (it does not matter which), say $\sigma_2$, and \MAYZR{as last} in
    $\sigma_3$. This way, we obtain a good triple of $D$.
	
	\main*
	
	\begin{proof}
		
        \MA{By the comment made before the proof, we may assume that $H$ is
        connected, $2$-regular and contains no transitive triangle.} Let $x
        \in V(H)$ be arbitrary and let $N^-_H(x) = \{a_1,a_2\}$ and
        $N^+_H(x) = \{b_1,b_2\}$. Let $D = H-x$. Recall that our goal is to
        find a good triple $T = (\sigma_1,\sigma_2,\sigma_3)$ of $D$ such
        that $a_1$ and $a_2$ lie before $b_1$ and $b_2$ in $\sigma_1$.
		
		Let $x_1$ be the unique out-neighbor of $a_2$ in $D$. We now construct an
		anti-directed path starting with the arc $a_2x_1$. By extending an
		anti-directed path $P$ we mean appending a vertex $u$ to $P$ such that
		$Pu$ is an anti-directed path. Obtain $P = a_2x_1x_2\dots x_\ell$ by
		extending the anti-directed path $a_2x_1$ until
		\begin{enumerate}[label=(\alph*)]
			\item\label{cond:max} there exists no vertex $u \in V(D)$ such
			that $Pu$ is an anti-directed path,
			\item\label{cond:a} $x_\ell = a_1$, or
			\item\label{cond:b} $x_\ell \in \{b_1,b_2\}$.
		\end{enumerate}
		
		We consider the case where $P$ satisfies conditions \ref{cond:max},
		\ref{cond:a}, and \ref{cond:b} separately. The three cases are
		illustrated in Figure \ref{fig:cases}.
		
		\newcommand{\base}{
			\fontsize{10}{10}\selectfont
			\tikzset{antipath/.style={<->, line width=0.03cm, decorate,decoration={snake,amplitude=.3mm,segment
						length=2mm,post length=1mm,pre length=1mm}}}
			\coordinate (a1pos) at (0,1.5);
			\coordinate (a2pos) at (1.5,1.5);
			\coordinate (xpos) at (0.75,0.75);
			\coordinate (b1pos) at (0,0);
			\coordinate (b2pos) at (1.5,0);
			\node[vertexB] (a1) at (a1pos) {$a_1$};
			\node[vertexB] (a2) at (a2pos) {$a_2$};
			\node[vertexB] (x) at (xpos) {$x$};
			\node[vertexB] (b1) at (b1pos) {$b_1$};
			\node[vertexB] (b2) at (b2pos) {$b_2$};
			
			\draw[arc,dashed] (a1) -- (x);
			\draw[arc,dashed] (a2) -- (x);
			\draw[arc,dashed] (x) -- (b1);
			\draw[arc,dashed] (x) -- (b2);
		}
		
		\begin{figure}[h]
            \renewcommand\thesubfigure{\arabic{subfigure}}
			\centering
			
			\begin{subfigure}[b]{0.3\textwidth}
				\centering
				\begin{tikzpicture}[node distance=1.5cm]
					\base
					\node[vertexB] (x1) [right of=a2] {$x_1$};
					\node[vertexB] (xl) [right= 0.8cm of x1] {$x_\ell$};
					
					\draw[arc] (a2) -- (x1);
					\draw[antipath] (x1) to (xl);
					
				\end{tikzpicture}
				\caption{$P$ satisfies \ref{cond:max}}
			\end{subfigure}
			\begin{subfigure}[b]{0.3\textwidth}
				\centering
				\begin{tikzpicture}[node distance=1.5cm]
					\base
					
					\node[vertexB] (x1) [above of=a2] {$x_1$};
					\draw[bend right,antipath] (x1) to (a1);
					
					\draw[arc] (a2) -- (x1);
				\end{tikzpicture}
				\caption{$P$ satisfies \ref{cond:a}}
			\end{subfigure}
			\begin{subfigure}[b]{0.3\textwidth}
				\centering
				\begin{tikzpicture}[node distance=1.5cm]
					\base
					\node[vertexB] (x1) [right of=a2] {$x_1$};
					\draw[bend left, antipath] (x1) to (b2);
					
					\draw[arc] (a2) -- (x1);
				\end{tikzpicture}
				\caption{$P$ satisfies \ref{cond:b} ($x_\ell = b_2$)}
			\end{subfigure}
			\caption{The three cases of $P$. Bidirectional squiggly edges symbolize
				anti-directed paths.}
			\label{fig:cases}
		\end{figure}
		
		In each case, we obtain the
		desired good triple of $D$. In all cases, let \[
            D'=D-(V(P)\setminus\{x_\ell\})
		\]
		Furthermore, note that $|V(P)| \geq 3$, as there are no transitive
        triangle in $H$ and $H$ is $2$-regular.
		Furthermore, it is easy to show by contradiction
		that no subgraph of $D$ contains a $2$-regular component, and thus we
		may apply Lemma \ref{lem:nonregular}.
		
		\quad 
		
		\noindent\textbf{Case 1}: There exists no vertex $u \in V(D)$ such that
		$Pu$ is an anti-directed path.
		
		Since $a_1$ is unbalanced in $D-V(P)$, by Lemma \ref{lem:nonregular},
		there exists a good $a_1$-triple $T' = (\pi^{a_1},\pi_{a_1},\pi)$ of
		$D-V(P)$. 
		
		We now insert $x_\ell$ into $T'$. 
		Suppose $x_{\ell-1}x_\ell \in A(P)$. Then $x_\ell$ cannot have an
		in-neighbor $u$ in $V(D')$ since then $Pu$ is a longer anti-directed
		path in $D$. Thus, the triple $T = (x_\ell\pi^{a_1},\pi_{a_1}x_\ell,
		x_\ell\pi)$ is a good $x_\ell$-triple of $D'$.
		Suppose $x_{\ell}x_{\ell-1} \in A(P)$. Then $x_\ell$ cannot have an
		out-neighbor $u$ in $V(D')$ since then $Pu$ is a longer anti-directed
		path in $D$. Thus, the triple $T = (x_\ell\pi^{a_1},\pi_{a_1}x_\ell,
		\pi x_\ell)$ is a good $x_\ell$-triple of $D'$. In any case, $T$ is a
		good $x_\ell$-triple $(\alpha^{x_\ell}, \alpha_{x_\ell},\alpha)$ of
		$D'$ with $\pi^{a_1} \leq \alpha^{x_\ell}$. Now, by Lemma
		\ref{lem:antidirected}, there exists a good $a_2$-triple
		$(\sigma^{a_2},\sigma_{a_2},\sigma)$ of $D$ such that $\pi^{a_1} \leq
		\sigma^{a_2}$. Since $a_1$ is before $b_1$ and $b_2$ in $\pi^{a_1}$,
		$\pi^{a_1} \leq \sigma^{a_2}$, and $a_2$ is first in $\sigma^{a_2}$, we
		have that both $a_1$ and $a_2$ are before $b_1$ and $b_2$ in
		$\sigma^{a_2}$ which is what we wanted. This completes the proof of
		Case 1.
		
		\quad
		
		\noindent \textbf{Case 2}: $x_\ell = a_1$.
		
		This is the simplest case.  We obtain a good $a_1$-triple
		$(\pi^{a_1},\pi_{a_1},\pi)$ of $D'$ by Lemma
		\ref{lem:nonregular} since $a_1$ is unbalanced in $D'$. By Lemma
		\ref{lem:antidirected}, there exists a good $a_2$-triple
		$(\sigma^{a_2},\sigma_{a_2},\sigma)$ of $D$ such that $\pi^{a_1}
		\leq \sigma^{a_2}$. Since $a_1$ is before $b_1$ and $b_2$ in
		$\pi^{a_1}$ and $a_2$ is first in $\sigma^{a_2}$, we have again
		obtained the desired good triple of $D$.
		
		\quad
		
		\noindent \textbf{Case 3}: $x_\ell \in \{b_1,b_2\}$.
		
		Assume without loss of generality that $x_\ell = b_2$. This case is
		more involved than the previous cases. We consider three subcases
		depending on the arcs incident with $b_2$. The subcases are illustrated
		in Figure \ref{fig:casesb}.
		
		\begin{figure}[h]
			\centering
			\begin{subfigure}[b]{0.24\textwidth}
				\centering
				\begin{tikzpicture}[node distance=1.5cm]
					\base
					\node[vertexB] (xlm1) [right of= b2] {$x_{\ell-1}$};

					\draw[bend left, antipath] (a2) to node[below left] {$P$}
					(xlm1);
					
					\draw[arc] (xlm1) -- (b2);
				\end{tikzpicture}
				\caption{Subcase 3a}
			\end{subfigure}
			\begin{subfigure}[b]{0.24\textwidth}
				\centering
				\begin{tikzpicture}[node distance=1.5cm]
					\base
					\node[vertexB] (xlm1) [right of= b2] {$x_{\ell-1}$};

					\draw[bend left, antipath] (a2) to node[below left] {$P$}
					(xlm1);
					
					\draw[arc] (b2) -- (xlm1);
					
					\draw[arc, dotted, thick] (b2) to[bend right] (a1);
				\end{tikzpicture}
				\caption{Subcase 3b}
			\end{subfigure}
			\begin{subfigure}[b]{0.24\textwidth}
				\centering
				\begin{tikzpicture}[node distance=1.5cm]
					\base
					
					\node[vertexB] (xlm1) [right of= b2] {$x_{\ell-1}$};
					\node[vertexB] (s1) [below of=b2] {$s_1$};
					\node[vertexB] (sk) [below of=b1] {$s_k$};

					\draw[bend left, antipath] (a2) to node[below left] {$P$}
					(xlm1);
					\draw[antipath] (s1) to node[above right] {$Q$} (sk);
					
					\draw[arc] (b2) -- (xlm1);
					\draw[arc] (b2) -- (s1);
					
				\end{tikzpicture}
                \renewcommand\thesubfigure{c1}
				\caption{Subsubcase 3c1}
			\end{subfigure}
			\begin{subfigure}[b]{0.24\textwidth}
				\centering
				\begin{tikzpicture}[node distance=1.5cm]
					\base
					
					\node[vertexB] (xlm1) [right of= b2] {$x_{\ell-1}$};
					\node[vertexB] (s1) [below of=b2] {$s_1$};

					\draw[bend left,antipath] (a2) to node[below left] {$P$}
					(xlm1);
					\draw[bend left, antipath] (s1) to node[above right]
					{$Q$} (b1);
					
					\draw[arc] (b2) -- (xlm1);
					\draw[arc] (b2) -- (s1);
					
				\end{tikzpicture}
                \renewcommand\thesubfigure{c2}
				\caption{Subsubcase 3c2}
			\end{subfigure}
			\caption{The subcases of Case 3. Bidirectional squiggly edges
				symbolize anti-directed paths.}
			\label{fig:casesb}
		\end{figure}

		\quad
		
		\noindent \textbf{Subcase 3a}: $x_{\ell-1} b_2 \in A(P)$.
		
		By Lemma \ref{lem:nonregular}, let $(\pi^{a_1},\pi_{a_1},\pi)$ be good
		$a_1$-triple of $D-\YZ{V(P)}$. The in-neighbors of $b_2$ in $H$ are $x$ and
		$x_{\ell-1}$, and thus $b_2$ has no in-neighbors in $D'$ and hence the
		$b_2$-triple $T = (b_2\pi_{a_1},\pi^{a_1}b_2,b_2\pi)$ of $D'$ is good.
		By Lemma \ref{lem:antidirected}, there exists a good triple
		$(\sigma^{a_2},\sigma_{a_2},\sigma)$ of $D$ such that $\pi^{a_1}b_2
		\leq \sigma^{a_2}$, which is what we wanted.
		
		\quad
		
		\noindent \textbf{Subcase 3b}: $b_2x_{\ell-1} \in A(P)$ and $b_2$ has
		no out-neighbor distinct from $a_1$ in $D'$.
		
		By Lemma \ref{lem:nonregular}, let $T' = (\pi^{a_1},\pi_{a_1},\pi)$ be
		a good $a_1$-triple of $D-\YZ{V(P)}$. 
		
		Let $T = (b_2\pi,\pi^{a_1}b_2,\pi_{b_2a_1})$ where $\pi_{b_2a_1}$ is
		the ordering obtained from $\pi_{a_1}$ by inserting $b_2$ immediately
		before $a_1$. Then $\pi_{b_2a_1}$ has no backward arcs incident with
		$b_2$ since $b_2$ has no out-neighbor distinct from $a_1$. We use here
		that $a_1\YZ{b_2} \notin A(D')$ since this would imply that $\{a_1,b_2,x\}$ induces a transitive triangle in $H$. The backward arcs incident with
		$b_2$ with respect to $b_2\pi$ are exactly the arcs entering $b_2$ and
		the backward arcs incident with $b_2$ with respect to $\pi^{a_1}b_2$ are
		exactly the arcs leaving $b_2$. Thus, $T$ is a good $b_2$-triple of
		$D'$. We can now apply Lemma \ref{lem:antidirected} to obtain a good
		$a_2$-triple $(\sigma^{a_2},\sigma_{a_2},\sigma)$ of $D$ such that
		$\pi^{a_1}b_2 \leq \sigma^{a_2}$, and we have obtained the desired good
		triple of $D$.
		
		\quad
		
		\noindent \textbf{Subcase 3c}: $b_2x_{\ell-1} \in A(P)$ and $b_2$ has an out-neighbor
		$s_1 \neq a_1$ in $D'$.
		
		In this case, we extend the anti-directed path
		$b_2s_1$ in $D'$, obtaining $Q = b_2s_1\dots s_k$, such that
		
		\begin{enumerate}[label=(\alph*)]
			\item\label{cond:maxQ} there exists no $u \in V(D')$ such that $Qu$
			is an anti-directed path in $D'$ or
			\item\label{cond:skQ} $s_k \in \{a_1,b_1\}$.
		\end{enumerate}

		Let $D'' = D'-(V(Q) \setminus \{s_k\})$. We consider the cases where $Q$ satisfies
		conditions \ref{cond:maxQ} and \ref{cond:skQ} separately. In each case,
		we obtain a good $b_2$-triple $T = (\sigma^{b_2}, \sigma_{b_2},\sigma)$
		of $D'$ such that $a_1$ is before $b_1$ and $b_2$ in $\sigma_{b_2}$.
		This is sufficient since then, by Lemma \ref{lem:antidirected} applied
		to $D'$, $P$, and $T$, we obtain a good $a_2$-triple $(\alpha^{a_2},
		\alpha_{a_2},\alpha)$ of $D$ such that $\sigma_{b_2} \leq \alpha^{a_2}$, which is what we want.

		\quad
		
		\noindent \textbf{Subsubcase 3c1}: There exists no $u \in V(D')$ such that $Qu$
		is an anti-directed path.
		
        By Lemma \ref{lem:nonregular}, let $\MAYZR{T''} = (\pi^{a_1},\pi_{a_1},\pi)$ be
        a good $a_1$-triple of $D''$. We insert $Q$ into $\MAYZR{T''}$ to obtain the
		desired good $b_2$-triple $T$ of $D'$. By the definition of this
		subsubcase, $s_k$ either has no in-neighbors or no out-neighbors in
		$D''$. If $s_k$ has no in-neighbors in $D''$, then let $T' =
		(s_k\pi^{a_1},\pi_{a_1}s_k,s_k\pi)$. If $s_k$ has no out-neighbors in
		$D''$, then let $T' = (s_k\pi^{a_1},\pi_{a_1} s_k,\pi s_k)$. In either
		case, $T'$ is a good $s_k$-triple of $D''$. 
		
		If $|V(Q)| \geq 3$ then by Lemma \ref{lem:antidirected} applied to $Q$,
		we obtain a good $b_2$-triple $T = (\sigma^{b_2},\sigma_{b_2},\sigma)$
		of $D'$ such that $s_k\pi^{a_1} \leq \sigma_{b_2}$ as desired.
		
		If $|V(Q)| = 2$, then we manually insert $s_1$ and $b_2$ into $T'$.
		Observe $s_k = s_1$ which is an out-neighbor of $b_2$. Thus, $s_1$
		cannot have any in-neighbors in $D''$ since then $Q$ could be extended.
		Therefore, the triple $(s_1\YZ{\pi^{a_1}},\pi_{a_1}s_1,\YZ{s_1\pi})$ is good.
		Furthermore, $b_2$ has no out-neighbor distinct from $s_1$ in $D'$
		since $b_2x_{\ell-1} \in A(P)$ in this subcase. Thus, the $b_2$-triple
		\YZ{$(b_2 s_1\pi, s_1\pi^{a_1}b_2, \pi_{a_1}b_2s_1$)} is good and we have
		again obtained the desired good triple of $D'$.
		
		This completes the proof of Subsubcase 3c1.
		
		\quad
		
		\noindent \textbf{Subsubcase 3c2}: $s_k \in \{a_1,b_1\}$.
		
		Let $T' = (\pi^{s_k},\pi_{s_k},\pi)$ be a good $s_k$ triple of $D''$.
		If $s_k = a_1$, then $a_1$ is before $b_1$ in $\pi^{s_k}$ and if $s_k =
		b_1$, then $a_1$ is before $b_1$ in $\pi_{s_k}$. Thus, $T'$ is a good
		$s_k$-triple of $D''$ such that $a_1$ is before $b_1$ in either
		$\pi^{s_k}$ or $\pi_{s_k}$. Let $\pi^* \in \{\pi^{s_k},\pi_{s_k}\}$ be
		such that $a_1$ is before $b_1$ in $\pi^*$.
		We have $|V(Q)| \geq 3$ since $s_1 \neq a_1$ and also $s_1$ cannot be
		$b_1$ since then $\{b_1,b_2,x\}$ induces a transitive triangle in $H$.
		Thus, by \YZ{Lemma \ref{lem:antidirected}}, we obtain a good $b_2$-triple $T =
		(\sigma^{b_2},\sigma_{b_2},\sigma)$ of $D'$ such that $\pi^* \leq
		\sigma_{b_2}$. 
		
		This completes the proof of Subsubcase 3c2.
		
		\quad
		
		We now return to the proof of the theorem. In all of the above cases,
		we obtain a good $a_2$-triple $T' = (\pi^{a_2},\pi_{a_2},\pi)$ of $D$
		such that $a_2$ and $a_1$ are before $b_1$ and $b_2$ in $\pi^{a_2}$.
		Now, obtain $\sigma$ by inserting $x$ into $\pi^{a_2}$ after $a_2$ and
		$a_1$ but before $b_1$ and $b_2$. Then, no arcs incident with $x$ are
		\YZ{backward arcs} with respect to $\sigma$ and thus $T = (x\YZ{\pi},\pi_{a_2}x,\sigma)$ is
		a good triple of $H$ as desired. 
	\end{proof}
	
	\section{Proving $\fasd(3,g)=g$ for $3\le g\le 5$ and  $\fas(3,6)=1/6$}\label{sec: fasd(3,4)=4}
    \label{sec:d=3}

	\begin{restatable}{thm}{girthIII}  \label{thm:girthIVdegIII}
$\fasd(3,g)=g$ for $g\in \{3,4,5\}$.
	\end{restatable}
	
	\begin{proof}
		 \YZ{Observe that the theorem is equivalent to the statement that if $g\in\{3,4,5\}$, then every orgraph $D \in \YZ{{\cal D}_{3,g}}$ admits a good $g$-arc-coloring.}
		Assume that the theorem is false and that $g \in \{3,4,5\}$ and $D \in \YZ{{\cal D}_{3,g}}$ is a graph of girth at least $g$ of minimum possible order,
		such that $D$ has no good $g$-arc-coloring.
		We will now obtain a contradiction, thereby completing the proof.
		By the minimality of the order of $D$ we may assume that every subgraph $D'$ of $D$ with fewer vertices than $D$ has a good $g$-arc-coloring.
		We first prove the following claims.

		\2
	
		{\bf Claim~A:} \YZ{$D$ is strongly connected. In particular, $\delta^+(D) \geq 1$ and $\delta^-(D) \geq 1$.}
		
		
        \begin{proof}[Proof of Claim A]
            \renewcommand{\qedsymbol}{$\diamond$}
        For the sake of contradiction assume $D$ is not strongly connected.
		By the minimality of $|V(D)|$ we can partition the arc set in each strong component of $D$ into $g$ feedback arc sets.
		Merging these we obtain $g$ arc-disjoint feedback arc sets of $D$, a contradiction. \YZ{Thus, $D$ is strongly connected. Now, if $d^+(x)=0$ or $d^-(x)=0$, then $x$ is a strongly connected component and therefore $V(D)=\{x\}$, which contradicts the fact that $D$ has no good $g$-arc-coloring.} This proves Claim~A.
        \end{proof}

		
		{\bf Claim~B:} Let $X_{12} = \{ x \; | \; d^+(x)=1, \; \; d^-(x)=2 \}$ and
		$X_{21} = \{ x \; | \; d^+(x)=2, \; \; d^-(x)=1 \}$ and
		$X_{11} = \{ x \; | \; d^+(x)=1, \; \; d^-(x)=1 \}$. Then $V(D)=X_{12} \cup X_{21} \cup X_{11}$ and $|X_{21}|=|X_{12}| \geq 1$.

		

        \begin{proof}[Proof of Claim B]
            \renewcommand{\qedsymbol}{$\diamond$}
        The fact that $V(D)=X_{12} \cup X_{21} \cup X_{11}$ follows from
        Claim~A and the fact that $D$ has maximum degree at most three. The
        following now holds:
		\[
		|V(D)| + |X_{21}| = \sum_{v \in V(D)} d^+(v) = a(D) = \sum_{v \in V(D)} d^-(v) = |V(D)| + |X_{12}|.
		\]
		
		Therefore $|X_{21}|=|X_{12}|$. First assume that $X_{12} = \emptyset$, which implies that $X_{21} = \emptyset$ and $V(D)=X_{11}$.
		\MAYZR{Then, by Claim A, $D$ is a cycle with length at least $g$}, as the girth of $D$ is at least $g$. 
		So we can \YZ{obtain a $g$-good-coloring of $D$ by coloring at least one arc in every cycle with each color, a contradiction.} So, $|X_{21}|=|X_{12}| \geq 1$. \YZ{This proves Claim B.}
        \end{proof}
		
		
		\YZ{Thus, by Claim~A and B there exists a $(X_{12},X_{21})$-path in $D$.}

                \2
		
		{\bf Claim~C:} 
		If $P=p_1 p_2 p_3 \cdots p_l$ is a shortest $(X_{12},X_{21})$-path in $D$ then $l \leq g-2$.

		

        \begin{proof}[Proof of Claim C]
            \renewcommand{\qedsymbol}{$\diamond$}
		Let $P=p_1 p_2 p_3 \cdots p_l$ be a shortest $(X_{12},X_{21})$-path in $D$ and assume that $l \geq g-1$ and note that $\{p_2,p_3,\ldots,p_{l-1}\} \subseteq X_{11}$.
		Let $D' = D-V(P)$ and note that $D'$ has a good $g$-arc-coloring. Take such a coloring and color all arcs into $p_1$ with color $1$, color $p_{i-1} p_{i}$ with color $i$ for all $i=2,3,\ldots,g-1$ and color all arcs out of $p_l$ with color $g$ (any remaining arcs can be colored arbitrarily \YZ{if $l\geq g$}). This coloring is a good coloring of $D$, contradicting the fact that $D$ does not have such a coloring, and thereby proving Claim~C.
        \end{proof}
		
		
	\MAYZR{Note that, by Claim~C, we have $g \ge 4$, since if $g = 3$, then $l \le 1$, which contradicts the fact that $P$ is a $(X_{12},X_{21})$-path. }
	
	\2
		
		{\bf Claim~D:} 
		If $P=p_1 p_2 p_3 \cdots p_l$ is defined as in Claim~C, then $l \leq g-3$.

		
        \begin{proof}[Proof of Claim D]
            \renewcommand{\qedsymbol}{$\diamond$}
        \MAYZR{Assume, for the sake of contradiction, that $l \ge g-2$. Then, by Claim~C, we must have $l = g-2$. As $g\geq 4$, $l\in \{2,3\}$.}
		Let $N^-(p_1) = \{w_1,w_2\}$ and consider the following three cases, which complete the proof of Claim~D.
		
        \begin{figure}[h]
            \centering
            \newcommand{\baseD}{
                    \node[vertexB] (p1) {$p_1$};
                    \node[vertexB] (p2) [right of = p1] {$p_2$};
                    \node[vertexB] (p3) [right of = p2] {$p_3$};

                    \node[vertexB] (w1) [above left of=p1] {$w_1$};
                    \node[vertexB] (w2) [below left of=p1] {$w_2$};

                    \draw[arc] (w1) -- (p1);
                    \draw[arc] (w2) -- (p1);
                    \draw[arc] (p1) -- (p2);
                    \draw[arc] (p2) -- (p3);

                    \draw[arc] (p3) -- +(0.5,0.3);
                    \draw[arc] (p3) -- +(0.5,-0.3);
                }

            \begin{subfigure}{0.32\textwidth}
                \centering
                \begin{tikzpicture}[node distance=1.5cm]
                	 
                    \baseD
                    \draw[arc,<-] (w1) -- +(-0.5,0.3);
                    \draw[arc,<-,dotted] (w1) -- +(-0.5,-0.3);

                    \draw[arc,<-] (w2) -- +(-0.5,0.3);
                    \draw[arc,<-,dotted] (w2) -- +(-0.5,-0.3);
                \end{tikzpicture}
                \subcaption{$w_1,w_2 \in X_{12} \cap X_{11}$}
            \end{subfigure}
            \begin{subfigure}{0.32\textwidth}
                \centering
                \begin{tikzpicture}[node distance=1.5cm]
                	
                    \baseD
                    \node[vertexB] (z1) at (-1.2,0.95) {$z_1$};
                      \node[vertexB] (z2) at (-1.2,-0.2) {$z_2$};
                    \draw[arc,<-] (w1) -- +(z1);
                    \draw[arc] (w1) -- +(0.5,0.3);

                   \draw[arc,<-] (w2) -- (z2);
                   \draw[arc] (w2) -- +(0.5,-0.3);
                \end{tikzpicture}
                \subcaption{$w_1,w_2 \in X_{21}$}
            \end{subfigure}
            \begin{subfigure}{0.32\textwidth}
                \centering
                \begin{tikzpicture}[node distance=1.5cm]
                    \baseD
                         \node[vertexB] (z2) at (-1.2,-0.2) {$z_2$};
                    \draw[arc,<-] (w1) -- +(-0.5,0.3);
                    \draw[arc,<-,dotted] (w1) -- +(-0.5,-0.3);

                      \draw[arc,<-] (w2) -- (z2);
                    \draw[arc] (w2) -- +(0.5,-0.3);
                \end{tikzpicture}
                \subcaption{$w_1 \in X_{11} \cup X_{12}, w_2 \in
                X_{21}$}
            \end{subfigure}
            \caption{The three cases in the proof of Claim D when $l=3$. Dotted arcs
            can be absent.}
            \label{fig:claimDcases}
        \end{figure}
		
		{\bf Case (a), $w_1,w_2 \in X_{12} \cup X_{11}$.}  
		Let \MAYZR{$D' = D-\{w_1,w_2,p_1,\dots, p_l\}$ (recall that $l\in \{2,3\}$)} and note that $D'$ has a good $g$-arc-coloring.
		Take such a coloring and color all arcs into $w_1$ or $w_2$ with color $1$, color $w_1 p_1$  and $w_2 p_1$ with color $2$, the arc $p_1 p_2$ with color $3$ and color all arcs out of $p_2$ with color $4$. Finally if $g=5$ (and $l=3$) then color all arcs out of $p_3$ with color $5$. This coloring is a good coloring of $D$, contradicting the fact that $D$ does not have such a coloring.
		
		\2
		
		{\bf Case (b), $w_1,w_2 \in X_{21}$.} In this case let $z_i w_i$ be the arc into $w_i$ in $D$ for $i=1,2$ and let 
		Let $D' = D-\{p_1,p_2\}$ and note that $D'$ has a good $g$-arc-coloring.
		Let $c$ be such a $g$-arc-coloring of $D'$ and, by possibly permuting the colors, we may without loss of generality assume that $c(z_1w_1)=1$ and $c(z_2w_2) \in \{1,2\}$.
		We now color $w_1 p_1$ with color 2 and we color $w_2 p_1$ with color $3-c(z_2w_2)$.
		We then color the arc $p_1 p_2$ with color $3$ and we color all arcs out of $p_2$ with color $4$.
		Finally if $g=5$ (and $l=3$) then color  
		all arcs out of $p_3$ with color $5$. 
		This coloring is a good coloring of $D$,
		contradicting the fact that $D$ does not have such a coloring.
		
		\2
		
		{\bf Case (c), $w_i \in X_{12} \cup X_{11}$ and $w_{3-i} \in X_{21}$, for some $i\in \{1,2\}$.}  We can without loss of generality assume that $i=1$. We proceed analogously to Case~1 and Case~2. Let $z_2 w_2$  be the arc into $w_2$ in $D$.
		
		Let \MAYZR{$D' = D-\{w_1,p_1,\dots,p_l\}$} and note that $D'$ has a good $g$-arc-coloring. 
		Let $c$ be such a $g$-arc-coloring of $D'$ and, by possibly permuting the colors, we may without loss of generality assume that $c(z_2w_2)=1$.
		We now color all arcs into $w_1$ with color 1 and we color $w_1 p_1$  and $w_2 p_1$ with color $2$ and we color arc $p_1 p_2$ 
		with color $3$ and we color all arcs out of $p_2$ with color $4$. 
		Finally if $g=5$ (and $l=3$) then color
		all arcs out of $p_3$ with color $5$. 
		This coloring is a good coloring of $D$,
		contradicting the fact that $D$ does not have such a coloring.
		This completes the proof of Claim~D.
        \end{proof}
		

		
		We now return to the proof of the theorem.  By Claim~D we note that $g=5$ and $l=2$. \MAYZR{Thus, let $p_1 p_2$ be an arc from $X_{12}$ to $X_{21}$.} Let $N^-(p_1) = \{w_1,w_2\}$ and let $N^+(p_2)=\{q_1,q_2\}$ and
		let $D^* = D - \{p_1,p_2\}$.
		Recall that a $5$-arc coloring of a digraph is good if every cycle contains arcs of all $5$ colors.  
		We note that an equivalent definition is that every color induces a feedback arc set.
		By the minimality of $|V(D)|$ we note that $D^*$ has a good $5$-arc-coloring.
		
		Let $A^-(x)$ denote all arcs into a vertex $x$ and let $A^+(x)$ denote all arcs out of $x$ \YZ{in $D^*$}.
		We call a $5$-arc-coloring of $D^*$ special if it is good and each of the sets $A^-(x)$ and $A^+(x)$ are both monochromatic for each
		$x \in WQ$, where $WQ=\{w_1,w_2,q_1,q_2\}$. 
		If $c$ is a special $5$-arc-coloring of $D^*$ then $c(A^+(x))$ denotes the unique color of the arcs in $A^+(x)$ (if $A^+(x) \not= \emptyset$) and 
		$c(A^-(x))$ denotes the unique color of the arcs in $A^-(x)$ (if $A^-(x) \not= \emptyset$) for all $x \in WQ$.
		
		We now prove the following claims.
		
		\2
		
		{\bf Claim~E:} $D^*$ contains a special $5$-arc-coloring.
		
		
        \begin{proof}[Proof of Claim E]
            \renewcommand{\qedsymbol}{$\diamond$}
        By the minimality of $|V(D)|$ we note that $D^*$ has a good $5$-arc-coloring, $c$.
		Let $X$ denote all vertices in $WQ$ that do not belong to any cycle in $D^*$ and recolor all arcs incident with vertices
		in $X$ by the color 1. Clearly this new coloring is still good as we have not recolored any arc that belongs to a cycle.
		Also for any $x \in WQ$ either $x \in X$ which implies that  $A^+(x)$ and  $A^-(x)$ are both monochromatic (with color 1) or
		$x \not\in X$, which implies that $d_{D^*}^+(x)=d_{D^*}^-(x)=1$,
        which again implies that  $A^+(x)$ and  $A^-(x)$ are both
        monochromatic (as both are sets of size one). Therefore, the new coloring is special. 
        \end{proof}

        {\bf Claim~F:} \MAYZR{Let $R=r_1r_2r_3 \cdots r_l$ be a path in
            $D^*$ with $r_1 \not\in WQ$ and $r_l \not\in WQ$ and let $c$ be any
            special $5$-arc-coloring of $D^*$ . If $d^+(r_i)=d^-(r_i)=1$ for all
        $i=2,3,\ldots, l-1$, then} no matter how we permute the colors on $R$ we
        will still have a special $5$-arc-coloring of $D^*$. Furthermore, if any
        color appears more than once on $R$ we can take one of the arcs with this
        color and recolor it arbitrarily and still have a special $5$-arc-coloring
        of $D^*$.
		
		
        \begin{proof}[Proof of Claim F]
            \renewcommand{\qedsymbol}{$\diamond$}
        Clearly any cycle in $D^*$ containing any of the recolored arcs contain all recolored arcs and therefore contain arcs of all possible colors. Therefore the coloring remains good. Furthermore, as $r_1 \not\in WQ$ and $r_l \not\in WQ$, we note that the coloring remains special.
		
		Assume that some color appears more than once on the arcs of $P$ and that we recolor one of these arcs in $P$, then every cycle in $D^*$ will still contain
		arcs of all colors, so we still have a special $5$-arc-coloring of $D^*$.
        \end{proof}


		{\bf Claim~G:} $w_1$ and $w_2$ are non-adjacent, $q_1$ and $q_2$ are non-adjacent, $c(A^-(w_1)) \not= c(A^-(w_2))$ and \YZ{$c(A^+(q_1)) \not= c(A^+(q_2))$} for 
		all special $5$-arc-colorings, $c$, of $D^*$.
		
		

        \begin{proof}[Proof of Claim G]
            \renewcommand{\qedsymbol}{$\diamond$}
        \MAYZR{We first show that if $w_1$ and $w_2$ are adjacent, or if $c(A^-(w_1)) = c(A^-(w_2))$, then there exists a special $5$-arc-coloring $c$ of $D^*$ such that every cycle of $D$ containing $p_1$ uses the color $c(A^-(w_1))$ (even though the arcs in $A(D) \setminus A(D^*)$ remain uncolored).} First assume that $w_1$ and $w_2$ are adjacent and without loss of generality assume that $w_1 w_2 \in A(D)$.
		As $d_D^-(w_1)>0$ we note that $A^-(w_1)=\{z_1w_1\}$ for some $z_1 \in V(D^*)$. 
		If $d_D^-(w_2)=2$ then $w_2$ has out-degree zero in $D^*$ so we may recolor the arcs into $w_2$ with color $c(A^-(w_1))$ and still have a 
		special coloring as the arcs into $w_2$ do not belong to 
		$A^+(q_1)$ or $A^+(q_2)$ (as girth is at least 5). If $d_D^-(w_2) \not=2$ we do not recolor any arcs. Now we note that all
		cycles in $D$ that contain $p_1$ will use the color $c(A^-(w_1))$ (even if the arcs in $A(D) \setminus A(D^*)$ are not colored).
		
		Now if $c(A^-(w_1)) = c(A^-(w_2))$ we also note that all
		cycles in $D$ that contain $p_1$ will use the color $c(A^-(w_1))$.

        \begin{figure}[h]
            \centering
            \begin{subfigure}{0.49\textwidth}
                \centering
            \begin{tikzpicture}[node distance=1.5cm]
                \tikzset{gone/.style={opacity=0.3}}
                \node[vertexB, gone] (p1) {$p_1$};
                \node[vertexB, gone] (p2) [right of = p1] {$p_2$};

                \node[vertexB] (w1) [above left of=p1] {$w_1$};
                \node[vertexB] (w2) [below left of=p1] {$w_2$};

                \node[vertexB] (q1) [above right of=p2] {$q_1$};
                \node[vertexB] (q2) [below right of=p2] {$q_2$};

                \draw[arc,gone] (w1) -- (p1);
                \draw[arc,gone] (w2) -- (p1);
                \draw[arc,gone] (p1) -- (p2);
                \draw[arc,gone] (p2) -- (q1);
                \draw[arc,gone] (p2) -- (q2);

                \draw[arc] (w1) -- (w2);

                \node[vertexB] (z1) [left of=w1] {$z_1$};
                \draw[arc] (z1) -- node [above,scale=0.52] {$c_1$} (w1);

                \draw[arc,<-,dotted] (w2) -- node [above,scale=0.52] {recolor to $c_1$} +(-1.5,0);
            \end{tikzpicture}
            \subcaption{$w_1w_2 \in A(D)$.}
            \end{subfigure}
            \begin{subfigure}{0.49\textwidth}
                \centering
            \begin{tikzpicture}[node distance=1.5cm]
                \tikzset{gone/.style={opacity=0.3}}
                \node[vertexB, gone] (p1) {$p_1$};
                \node[vertexB, gone] (p2) [right of = p1] {$p_2$};

                \node[vertexB] (w1) [above left of=p1] {$w_1$};
                \node[vertexB] (w2) [below left of=p1] {$w_2$};

                \node[vertexB] (q1) [above right of=p2] {$q_1$};
                \node[vertexB] (q2) [below right of=p2] {$q_2$};

                \draw[arc,gone] (w1) -- (p1);
                \draw[arc,gone] (w2) -- (p1);
                \draw[arc,gone] (p1) -- (p2);
                \draw[arc,gone] (p2) -- (q1);
                \draw[arc,gone] (p2) -- (q2);

                \node[vertexB] (z1) [left of=w1] {$z_1$};
                \draw[arc] (z1) -- node [above,scale=0.52] {$c_1$} (w1);

                \draw[arc,<-,dotted] (w1) -- node [below, scale=0.52] {$c_1$} +(-0.6,-0.4);

                \draw[arc,<-,dotted] (w2) -- node [above, scale=0.52] {$c_1$} +(-0.6,0.3);
                \draw[arc,<-,dotted] (w2) -- node [below, scale=0.52] {$c_1$} +(-0.6,-0.3);
            \end{tikzpicture}
            \subcaption{$\exists$ special $c$ with $c(A^-(w_1)) =
            c(A^-(w_2))$.}
            \end{subfigure}
            \begin{subfigure}{\textwidth}
                \centering
            \begin{tikzpicture}[node distance=1.5cm]
                \tikzset{gone/.style={opacity=0.3}}
                \node[vertexB] (p1) {$p_1$};
                \node[vertexB] (p2) [right of = p1] {$p_2$};

                \node[vertexB] (w1) [above left of=p1] {$w_1$};
                \node[vertexB] (w2) [below left of=p1] {$w_2$};

                \node[vertexB] (q1) [above right of=p2] {$q_1$};
                \node[vertexB] (q2) [below right of=p2] {$q_2$};

                \draw[arc] (w1) -- node [above,scale=0.52] {$c_2$} (p1);
                \draw[arc] (w2) -- node [below,scale=0.52] {$c_2$} (p1);
                \draw[arc] (p1) -- node [above,scale=0.52] {$c_3$} (p2);
                \draw[arc] (p2) -- (q1);
                \draw[arc] (p2) -- (q2);

                \node[scale=0.52] (label) [right=0.1cm of p2] {$c_4$ and $c_5$};

                \node[vertexB] (z1) [left of=w1] {$z_1$};
                \draw[arc] (z1) -- node [above,scale=0.52] {$c_1$} (w1);

                \draw[arc,dotted] (w1) -- (w2);
                \draw[arc,<-,dotted] (w1) -- node [below, scale=0.52] {$c_1$} +(-0.6,-0.4);

                \draw[arc,<-,dotted] (w2) -- node [above, scale=0.52] {$c_1$} +(-0.6,0.3);
                \draw[arc,<-,dotted] (w2) -- node [below, scale=0.52] {$c_1$} +(-0.6,-0.3);

                \draw[arc,dotted] (q1) -- +(0.6,0.3);
                \draw[arc,dotted] (q1) -- +(0.6,-0.3);

                \draw[arc,dotted] (q2) -- +(0.6,0.3);
                \draw[arc,dotted] (q2) -- +(0.6,-0.3);
            \end{tikzpicture}
            \subcaption{In both cases, we obtain a good coloring of $D$ with the indicated structure.}
            \end{subfigure}
            \caption{\MAn{The coloring obtained in Claim G when (a)
            $w_1$ and $w_2$ are adjacent or (b) there exists a special
            coloring $c$ with $c(A^-(w_1)) = c(A^-(w_2))$.}}

            \label{}
        \end{figure}

		So in the above cases let $c_1 = c(A^-(w_1))$ and note that  that all cycles in $D$ that contain $p_1$ will use the color $c_1$. As $D$ is an orgraph, without loss of generality we may assume that $q_1q_2\notin A(D^*)$ \MAYZR{(otherwise, we could exchange the indices of $q_1$ and $q_2$)}. 
		
		We now consider the case when $c(A^+(q_1))=c_1$. 
		
		If $d_{D^*}^+(q_1)=\YZ{d_{D^*}^-(q_1)}=1$, we can swap the colors of the arc entering and leaving
		$q_1$ in $D^*$. It is not difficult to see that the resulting coloring is special (even though some arc entering $q_1$ may come from $w_1$ or $w_2$ \YZ{as $d^+_{D^*}(w_1)$ and $d^+_{D^*}(w_2)$ are at most 1}).
		Furthermore, if the color of the arc entering and leaving $q_1$ are both $c_1$, then we can just recolor the arc leaving $q_1$ with a color different from $c_1$ and still have a special coloring. 
		
		If we do not have $d_{D^*}^+(q_1)=\YZ{d_{D^*}^-(q_1)=1}$, then \YZ{$d_{D^*}^-(q_1)=0$ and therefore} we can just recolor all arcs out of $q_1$ with a color different from $c_1$ and still have a special coloring.  
		So in all cases we can obtain a special coloring with $c(A^+(q_1)) \not=c_1$.
		
		\MAYZR{Now we can also recolor the arcs so that $c(A^+(q_2)) \neq c_1$, using exactly the same operation. More specifically,
			\begin{itemize}
				\item If $d_{D^*}^+(q_2) = d_{D^*}^-(q_2) = 1$, we simply swap the colors of the arc entering $q_2$ and the arc leaving $q_2$. And if $q_2$ is in $D^*$ and the colors of both the entering and the leaving arcs are $c_1$, then we can recolor only the arc leaving $q_2$.
				
				\item If $d_{D^*}^-(q_2) \neq 1$ (and therefore $d_{D^*}^-(q_2) = 0$), then we may recolor all arcs leaving $q_2$ with a color different from $c_1$.
			\end{itemize}
			
			Since $q_1q_2 \notin A(D^*)$ and the directed girth of $D$ is at least $5$, the recoloring does not affect the colors of the arcs in 
			$c(A^+(q_1))$, $c(A^-(w_1))$, or $c(A^-(w_2))$.} We now consider the following two cases.
		
		If $c(A^+(q_1)) = c(A^+(q_2))$ then let $c_5=c(A^+(q_1)) = c(A^+(q_2))$ and let $\{c_1,c_2,c_3,c_4,c_5\}=\{1,2,3,4,5\}$ and color the arcs
		$w_1 p_1$ and $w_2p_1$ with color $c_2$, color $p_1 p_2$ with color
        $c_3$ and color the arcs $p_2 q_1$ and $p_2 q_2$ \MA{with} color $c_4$.
		We then obtain a good $5$-coloring of $D$, a contradiction.
		
		If $c(A^+(q_1)) \not= c(A^+(q_2))$, then let  $c_4=c(A^+(q_1))$ and $c_5 = c(A^+(q_2))$ and let $\{c_1,c_2,c_3,c_4,c_5\}=\{1,2,3,4,5\}$ and color the arcs
		$w_1 p_1$ and $w_2p_1$ with color $c_2$, color $p_1 p_2$ with color
        $c_3$, color the arcs $p_2 q_1$ with color $c_5$ and color $p_2
        q_2$ \MAn{with} color $c_4$. We then obtain a good $5$-coloring of $D$, a contradiction. 
		
        This completes the proof when $w_1$ and $w_2$ are adjacent or there
        exists a special coloring $c$ with $c(A^-(w_1)) = c(A^-(w_2))$.
        The only remaining case is when $q_1$ and $q_2$ are adjacent \MAn{or
            there exists a special coloring $c$ with $c(A^+(q_1)) =
        c(A^+(q_2))$,} which can be proved analogously. This completes
            the proof of Claim~G.
        \end{proof}
		
		
        {\bf Claim~H:} The \MA{sets} $A^-(w_1)$, $A^-(w_2)$, $A^-(q_1)$,
        $A^-(q_2)$, $A^+(w_1)$, $A^+(w_2)$, $A^+(q_1)$, $A^+(q_2)$ are
        \MA{pairwise disjoint}.
		
		
        \begin{proof}[Proof of Claim H]
            \renewcommand{\qedsymbol}{$\diamond$}
        It suffices to show that $WQ$ is an independent set of 4 vertices. By Claim~G we note that $w_1$ and $w_2$ are non-adjacent and
		$q_1$ and $q_2$ are non-adjacent.  For the sake of contradiction assume that $w_1$ and $q_1$ are adjacent. 
        As the girth is at least 5 we note that this implies that $w_1 q_1 \in A(D)$.  \YZ{By Claim A} there must therefore exist vertices $s_1$ and $s_2$ such that 
		$s_1 w_1 q_1 s_2$ is a path in $D^*$. Furthermore as the girth is at least 5 and $\{w_1,w_2\}$ and $\{q_1,q_2\}$ are independent sets we note that 
		$s_1 \not\in WQ$ and $s_2 \not\in WQ$. Let $c_1$, $c_2$ and $c_3$ be the colors of $s_1 w_1$, \YZ{$w_1 q_1$} and \YZ{$q_1 s_2$}, respectively. 
		We may assume that $c_1$, $c_2$ and $c_3$ are distinct colors by Claim~F.
		
		Let $c_4= c(A^-(w_2))$ and note that $c_4$ is distinct from $c_1$, $c_2$ and $c_3$ \YZ{(otherwise by Claim~F we can swap $c(s_1w_1)$, $c(w_1q_1)$ and $c(q_1s_2)$ to obtain a special coloring with $c(A^-(w_1)) =c(s_1w_1)= c(A^-(w_2))$ which contradicts Claim~G)}. 
		If  $c_4= c(A^+(q_2))$, then we will swap the color of the arcs in $A^+(q_2)$ and $A^-(q_2)$ (if they are  both non-empty) or recolor $A^+(q_2)$ 
		(if $A^-(q_2)$ is empty or if $c(A^+(q_2))=c(A^-(q_2))$), such that $c$ remains special and $c_4 \not= c(A^+(q_2))$. 
		Let $c_5= c(A^+(q_2))$ and note that
		$c_1,c_2,c_3,c_4,c_5$ are five distinct colors \YZ{as otherwise by Claim~F we can swap $c(s_1w_1)$, $c(w_1q_1)$ and $c(q_1s_2)$ to obtain a special coloring with $c(A^+(q_1))=c(q_1s_2) = c(A^+(q_2))$ which contradicts Claim~G}. 
		Color the arc $w_1 p_1$ with color $c_4$,
		color \YZ{$w_2 p_1$} with color $c_1$,
		color $p_1 p_2$ with color $c_2$,
		color $p_2 q_1$ with color $c_5$ and 
		color $p_2 q_2$ with color $c_3$.
		We then obtain a good $5$-coloring of $D$, a contradiction.
		
        So $w_1$ and $q_1$ are non-adjacent. We can analogously show that
        $w_i$ and $q_j$ are non-adjacent for all $i \in [2]$ and $j \in
        [2]$, which completes the proof of Claim~H.		
        \end{proof}

		{\bf Claim~I:} If we swap the colors of the arcs in $A^-(x)$ and $A^+(x)$ for any $x \in WQ$, in a special $5$-arc-coloring of $D^*$, then we still have a  special $5$-arc-coloring of $D^*$. Furthermore there exists a special $5$-arc-coloring of $D^*$ such that 
		$c(A^-(x)) \not= c(A^+(x))$ (when they are both defined) for all $x \in WQ$.
		
		\2
		
		\begin{proof}[Proof of Claim I]    \renewcommand{\qedsymbol}{$\diamond$}
			\YZ{By Claim~H, for every $x,y\in WQ$ such
        that $x\neq y$, changing colors of $A^-(x)$ and $A^+(x)$ does not
    affect the colors of $A^-(y)$ and $A^+(y)$.} The fact that $c(A^-(x)) \not= c(A^+(x))$ (when they are both defined) for all $x \in WQ$ now follows from 
		Claim~F (considering the path $R$ of length two with $x$ as a central vertex).
		And the fact that when we swap the colors of the arcs in $A^-(x)$ and $A^+(x)$ for any $x \in WQ$ we still have a special coloring also
		follows from Claim~F.
			\end{proof}
		\2
		
		{\bf Claim~J:} There exists a special $5$-arc-coloring, $c$, of $D^*$ such that  the following are distinct colors (when defined),
		
		\[
		c(A^-(w_1)), c(A^-(w_2)), c(A^+(w_1)), c(A^+(w_2))
		\]
		
		and the following are distinct colors (when defined),
		
		\[
		c(A^-(q_1)), c(A^-(q_2)), c(A^+(q_1)), c(A^+(q_2))
		\]
		
		Furthermore, $c$ can be chosen such that $c(A^-(w_1)), c(A^-(w_2)), c(A^+(q_1)), c(A^+(q_2))$ are four distinct colors.
		
		
        \begin{proof}[Proof of Claim J]
            \renewcommand{\qedsymbol}{$\diamond$}
       Assume, for the sake of contradiction, that 
		$c(A^-(w_1))$, $c(A^-(w_2))$, $c(A^+(w_1))$ and $c(A^+(w_2))$ are not distinct colors. By Claim~I we can recolor 
		arcs such that $c(A^-(w_1)) = c(A^-(w_2))$ contradicting Claim~G. Therefore 
		$c(A^-(w_1))$, $c(A^-(w_2))$, $c(A^+(w_1))$ and  $c(A^+(w_2))$ are distinct colors (when defined) and analogously we can show that
		$c(A^-(q_1))$, $c(A^-(q_2))$, $c(A^+(q_1))$ and  $c(A^+(q_2))$ are distinct colors (when defined).
		
		For the sake of contradiction assume that $c(A^-(w_1)), c(A^-(w_2)), c(A^+(q_1)), c(A^+(q_2))$ are not four distinct colors,
		which implies that $c(A^-(w_i))=c(A^+(q_j))$ for some $i \in [2]$ and $j \in [2]$. Without loss of generality assume that $i=j=1$.
		If $c(A^-(q_1))$ is defined then, by Claim~I, swap the colors of $A^+(q_1)$ and $A^-(q_1)$ and if $A^-(q_1)$ is empty then recolor
		$A^+(q_j)$ with a color different from $c(A^-(w_1))$. Now $c(A^-(w_1)) \not= c(A^+(q_1))$.
		
		If $c(A^-(w_2))=c(A^+(q_1))$, then consider the following cases. If $c(A^+(w_2))$ is defined then, by Claim~I, swap the colors of $A^+(w_2)$ and $A^-(w_2)$ and if $A^+(w_2)$ is empty then recolor $A^-(w_2)$ with a color different from $c(A^+(q_2))$, $c(A^-(w_1))$ \MAYZR{and $c(A^+(w_1))$}. 
		Now $c(A^-(w_2)) \not= c(A^+(q_1))$ and $c(A^-(w_2)) \not= c(A^-(w_1))$, \MAYZR{and $c(A^-(w_1))$, $c(A^-(w_2))$ and $c(A^+(w_1))$ remain distinct (when defined)}.
		
		If $c(A^-(w_2))=c(A^+(q_2))$, then consider the following cases. If $c(A^-(q_2))$ is defined then, by Claim~I, swap the colors of $A^+(q_2)$ and $A^-(q_2)$ and if $A^-(q_2)$ is empty then recolor $A^+(q_2)$ with a color different from $c(A^+(q_1))$, $c(A^-(w_1))$, $c(A^-(w_2))$ \MAYZR{and $c(A^-(q_1))$}. 
		Now $c(A^+(q_2))$ is different from $c(A^+(q_1))$, $c(A^-(w_1))$ and $c(A^-(w_2))$,  \MAYZR{and $c(A^+(q_1))$, $c(A^+(q_2))$ and $c(A^-(q_1))$ remain distinct (when defined)}.
		This completes the proof of Claim~J.
        \end{proof}
		
		
		
		We now return to the proof of the theorem again. Let $c$ be a special $5$-arc-coloring of $D^*$ such that the properties of Claim~J hold.
		
		Now color $w_1 p_1$ with color $c(A^-(w_2))$ and color $w_2 p_1$ with color \YZ{$c(A^-(w_1))$} and
		color $p_2 q_1$ with color $c(A^+(q_2))$  and color $p_2 q_2$ with color $c(A^+(q_1))$.
		Finally color $p_1 p_2$ with the color in $\{1,2,3,4,5\} \setminus \{c(A^-(w_1)), c(A^-(w_2)), c(A^+(q_1)), c(A^+(q_2))\}$.
		Now $c$ becomes a good coloring of $D$, contradicting 
		the fact that $D$ does not have such a coloring.
	\end{proof}
\MAYZR{In Section \ref{sec:conclusion}, Problem \ref{prob:336} will ask whether $\fasd(3,6) = 5$ or 6.  
}	
	\MAYZR{Note that if $\fasd(3,6) = 6$, then 
	$\fas(3,6) = \tfrac{1}{6}$. To provide some evidence for the possibility of $\fasd(3,6) = 6,$
	we conclude this section by proving that indeed $\fas(3,6) = \tfrac{1}{6}$, 
	using a recent result from~\cite{ALGYZ} on the size of minimum feedback 
	vertex sets.
	
	Recall that a \emph{minimum feedback vertex set} of a digraph $D$ is a 
	smallest subset $F \subseteq V(D)$ such that $D - F$ is acyclic. 
	Let $\mathcal{C}_o$ denote the class of digraphs obtained from an undirected odd cycle by replacing each edge with a directed 2-cycle. 
	\begin{thm}\cite{ALGYZ} \label{thm:fvs4} 
		Let $D$ be a connected digraph with $\Delta(D) \le 4$. Then minimum feedback vertex sets of $D$ have size at most $\frac{|V(D)|}{2}$ unless $D \in \mathcal{C}_o$.
	\end{thm}
	
	The result can be easily extended to directed multigraphs, since adding multiple arcs does not affect the size of minimum feedback vertex sets. Hence, the following corollary holds.
	
	\begin{cor}\label{cor1}
		Let $D$ be a connected directed multigraph with $\Delta(D) \le 4$. Then minimum feedback vertex sets of $D$ have size at most $\frac{|V(D)|}{2}$ unless $D \in \mathcal{C}_o$.
	\end{cor}
	
	We call a directed multigraph $D$ \emph{degree-$k$} if the underlying graph $U(D)$ of it, obtained from replacing each arc by an edge, is $k$-regular. We now ready to show $\fas(3,6)=\frac{1}{6}$. }
	
	\begin{thm}\label{thm: fas3-6}
		$\fas(3,6)=\frac{1}{6}$.
	\end{thm}
	\begin{proof}
		One can observe from a directed 6-cycle that $\fas(3,6)\geq
        \frac{1}{6}$. Thus, \MAYZR{it remains to be} shown that, for every orgraph $D$ with $\Delta(D)\leq 3$ and $g(D)\geq 6$, $\fas(D)\leq \frac{a(D)}{6}$. Suppose, for the sake of contradiction, that the result is false, and let $D$ be a counterexample with the minimum number of vertices. We first show that $D$ is strongly connected.
		
		If $D$ is not strongly connected, let $H$ be a strong component of $D$. By the minimality of $D$, both $H$ and $D - V(H)$ have feedback arc sets $F_1$ and $F_2$ with desired size, respectively. Thus, $F_1\cup F_2$ is a feedback arc set of $D$ containing at most $\frac{a(H) + a(D - V(H))}{6} \le \frac{a(D)}{6}$ arcs, a contradiction. Hence, $D$ must be strongly connected.
		
		Therefore, the vertex set $V(D)$ can be partitioned into three subsets:
		\[
		X^+ = \{v \in V(D) : d^+(v) = 2,\, d^-(v) = 1\}, \]
		\[X^- = \{v \in V(D) : d^+(v) = 1,\, d^-(v) = 2\}, \]
		and
		\[X_0 = \{v \in V(D) : d^+(v) = d^-(v) = 1\}.
		\]
		Thus,
		\[
		2|X^+| + |X^-| + |X_0|
		= \sum_{v \in V(D)} d^+(v)
		= \sum_{v \in V(D)} d^-(v)
		= |X^+| + 2|X^-| + |X_0|,
		\]
		which implies that $|X^+| = |X^-|$.
		
		If $|X^+| = |X^-| = 0$, then, since $D$ is strongly connected, it must be a directed cycle of length at least $6$, for which the result clearly holds. Thus, $|X^+| = |X^-| > 0$.
		
		Because $D$ is strongly connected, the subgraph $D[X_0]$ is a union of vertex-disjoint directed paths. Denote these paths by $P_1, P_2, \dots, P_t$, and for each $i \in [t]$, write
		\[
		P_i = v^i_1 v^i_2 \dots v^i_{|V(P_i)|}.
		\]
		Let $y_i$ denote the unique out-neighbour of $v^i_{|V(P_i)|}$ in $X^- \cup X^+$, and let $x_i$ denote the unique in-neighbour of $v^i_1$ in $X^- \cup X^+$. The following claim now holds.
		
		\2
		
		{\bf Claim A.} For every $i\in [t]$, $x_i\neq y_i$. 
		\begin{proof}[Proof of Claim A]
			\renewcommand{\qedsymbol}{$\diamond$}
			If $x_i = y_i$, then $D[V(P_i) \cup \{x_i\}]$ forms a directed cycle having only one out-neighbour or only one in-neighbour in $V(D) \setminus (V(P_i) \cup \{x_i\})$ (i.e. the in- or out-neighbour of $x_i\in V(D)\setminus V(P_i)$), contradicting the fact that $D$ is strongly connected. This completes the proof.
		\end{proof}
		
		\2
		{\bf Claim B. }If $x_i\in X^-$ or $y_i\in X^+$, then $|V(P_i)|\leq 2$. 
		\begin{proof}[Proof of Claim B]
			We prove the result for the case $x_i \in X^-$, as the argument for the other case is analogous. Suppose, for the sake of contradiction, that $|V(P_i)| \ge 3$. By the minimality of $D$, let $F$ be a feedback arc set of $D' = D - V(P_i) - \{x_i\}$ with $|F| \le \frac{a(D')}{6}$. Since $|V(P_i)| \ge 3$, we have $a(D') \le a(D) - 6$. 
			
			Observe that every directed cycle in $D$ containing vertices in $V(P_i)\cup \{x_i\}$ must also contain the arc $x_iv^i_1$ (as $x_i\in X^-$ and $x_i$ is the only vertex in $V(P_i)\cup \{x_i\}$ which has in-neighbours in $V(D')$). Therefore, $F \cup \{x_i v^i_1\}$ is a feedback arc set of $D$ with 
			\[
			|F \cup \{x_i v^i_1\}| \le \frac{a(D) - 6}{6} + 1 = \frac{a(D)}{6},
			\]
			a contradiction.
		\end{proof}
		
		\2
		{\bf Claim C.} For every $i\in[t]$, $x_i\notin X^-$ or $y_i\notin X^+$.
		\begin{proof}[Proof of Claim C]
			Suppose there exists a path $P_i$ such that $x_i \in X^-$ and $y_i \in X^+$. Let
			$
			D' = D - V(P_i) - \{x_i, y_i\}.
			$
			From Claim~A, B and the assumption $g(D) \ge 6$, one can observe that
			$
			a(D') = a(D) - 6 \text{ or } a(D) - 7,
			$
			and in particular, $a(D') \le a(D) - 6$.
			
			Let $F$ be a feedback arc set of $D'$ with $|F|\leq \frac{a(D')}{6}$. Then $F \cup \{x_i v^i_1\}$ is a minimum feedback arc set of $D$, since every directed cycle in $D$ containing a vertex from $V(P_i) \cup \{x_i,y_i\}$ must also contain the arc $x_i v^i_1$ (as $x_i\in X^-$ and $x_i$ is the only vertex in $V(P_i) \cup \{x_i, y_i\}$ having in-neighbours in $V(D')$). Moreover, its size satisfies
			\[
			|F \cup \{x_i v^i_1\}| \le \frac{a(D) - 6}{6} + 1 = \frac{a(D)}{6},
			\]
			a contradiction.
		\end{proof}
		
		\2
		{\bf Cliam D.}		$X^+$ and $X^-$ are independent sets of $D$.
		\begin{proof}[Proof of Claim D]
			We only prove that $X^+$ is independent, as the argument for $X^-$ is analogous. Suppose, for the sake of contradiction, that $X^+$ is not independent. Note that $D[X^+]$ is acyclic as if it has a directed cycle $C$ then by the definition of $X^+$ vertices on $C$ has no in-neighbours in $V(D)\setminus V(C)$, which contradicts the fact that $D$ is strongly connected. Now, let 
			$
			P = v_1v_2 \dots v_{|V(P)|}
			$
			be the longest path in $D[X^+]$, so $|V(P)| \ge 2$. Then
            $v_{\MAYZR{1}}$ has an in-neighbour $v' \in X^- \cup X_0$.  
			
			Let $D' = D - \{v_1, v_2, v'\}$. Since $g(D) \ge 6$, we have
			\[
			a(D') \le a(D) - 6,
			\]
			where $a(D') = a(D) - 6$ if $v' \in X_0$ and $a(D') = a(D) - 7$ if $v' \in X^-$.  
			
			Let $F$ be a minimum feedback arc set of $D'$. Then $F \cup \{v'v_1\}$ is a feedback arc set of $D$, since every directed cycle in $D$ containing a vertex from $\{v_1, v_2, v'\}$ must also contain the arc $v'v_1$ (as $v'\in X_0\cup X^-$ and $v'$ is the only vertex in $\{v_1, v_2, v'\}$ having in-neighbours in $V(D')$). Moreover, its size satisfies
			\[
			|F \cup \{v'v_1\}| \le \frac{a(D) - 6}{6} + 1 = \frac{a(D)}{6},
			\]
			a contradiction.
		\end{proof}
		
		\2
		{\bf Claim E.} (i) For every $i\in [t]$, $x_i\in X^+$ and $y_i\in X^-$;  (ii) For every $i\in [t]$, $y_ix_i\notin A(D)$.
		\begin{proof}[Proof of Claim E]
			(i): \noindent
			By Claim~C, it suffices to show that for every $i \in [t]$, we have $\{x_i, y_i\} \not\subseteq X^-$ and $\{x_i, y_i\} \not\subseteq X^+$. We prove only the latter, as the former case is analogous.  
			
			Suppose there exists some $s \in [t]$ such that $\{x_s, y_s\} \subseteq X^+$. Let $x_s^-$ denote the unique in-neighbour of $x_s$. Then, by Claim~D, we have $x_s^- \in X^- \cup X_0$.  
			
			Now, let
			$
			D' = D - V(P_s) - \{x_s, y_s, x_s^-\},
			$
			and let $F$ be a minimum feedback arc set of $D'$. Since $g(D) \ge 6$, one can observe that
			$
			a(D') \le a(D) - 6.
			$
			
			Note that $F \cup \{x_s^- x_s\}$ is a feedback arc set of $D$, since every directed cycle in $D$ containing a vertex from $V(P_s) \cup \{x_s^-, x_s, y_s\}$ must also contain the arc $x_s^-x_s$ (as $x_s^-$ is the only vertex in $V(P_s) \cup \{x_s^-, x_s, y_s\}$ having in-neighbours in $V(D')$ and $x_s$ is the only out-neighbour of $x_s^-$ in $D$ since $x_s^-\in X^-\cup X_0$). Moreover, its size satisfies
			\[
			|F \cup \{x_s^- x_s\}| \le \frac{a(D) - 6}{6} + 1 = \frac{a(D)}{6},
			\]
			a contradiction. This completes the proof of (i).
			
			(ii): Suppose there exists a path $P_i$ such that $y_i x_i \in A(D)$. By Claim E (i), $x_i\in X^+$ and $y_i\in X^-$. Let
			$
			D' = D - V(P_i) - \{x_i, y_i\},
			$
			and let $F$ be a minimum feedback arc set of $D'$. Since $g(D) \ge 6$, one can observe that
			$
			a(D') \le a(D) - 6.
			$
			
			Note that $F \cup \{y_i x_i\}$ is a feedback arc set of $D$, because every directed cycle in $D$ containing a vertex from $V(P_i) \cup \{x_i, y_i\}$ must include the arc $y_i x_i$ (as $x_i$ is the only vertex in $V(P_i) \cup \{x_i, y_i\}$ with out-neighbours in $V(D')$ and $y_i$ is the only in-neighbour of $x_i$ in $D$). Moreover, its size satisfies
			\[
			|F \cup \{y_i x_i\}| \le \frac{a(D) - 6}{6} + 1 = \frac{a(D)}{6},
			\]
			a contradiction. This completes the proof.
		\end{proof}
		
			Now construct a directed multigraph $D'$ from $D$ by removing $X_0$ and adding an arc from $x_i$ to $y_i$ for each $i \in [t]$ (Note that by Claim A, $D'$ has no loops and therefore it is indeed a directed multigraph). Denote by $A'$ the set of these newly added arcs. Clearly, $V(D') = X^- \cup X^+$, $D'$ is degree-$3$, and every vertex in $X^-$ (respectively, $X^+$) has in-degree $2$ (respectively, out-degree $2$) in $D'$.
		
		It therefore suffices to show that $D'$ has a feedback arc set $F'$ with $|F'| \le a(D')/6$. Indeed, from such an $F'$ we can construct a feedback arc set $F$ of $D$ as follows:  
		for each $a \in F'$,  
		\MAYZR{
			\begin{itemize}
				\item if $a \notin A'$, include $a$ in $F$;  
				\item if $a \in A'$, then it corresponds to an arc $(x_i, y_i)$, and we  include the arc $x_i v^i_1$ in $F$ instead.  
			\end{itemize}
		}
		It is straightforward to verify that $F$ is a feedback arc set of $D$, and that
		\[
		|F| = |F'| \le \frac{a(D')}{6} \le \frac{a(D)}{6}.
		\]
		Hence, it suffice to show that $D'$ admits a feedback arc set of size at most $a(D')/6$.
		
		By Claim~D, $D'$ is bipartite with partite sets $(X^-,X^+)$. Thus, as every vertex in $X^-$ has out-degree $1$ in $D'$ and every vertex in $X^+$ has in-degree $1$ in $D'$, all arcs from $X^-$ to $X^+$ form a perfect matching $M$ in $D'$. Moreover, by Claim E (i), $M\cap A'=\emptyset$ and therefore $M$ is a matching in $D$. 
		
		Let $D''$ be the digraph obtained from $D'$ by contracting all arcs in $M$. Note that if $xy\in M$, then by Claim E (i) and (ii) and the fact that $D$ is oriented, $xy$ is the only arc between $x$ and $y$ in $D'$. Thus, $D''$ has no loops and therefore $D''$ is a degree-$4$ directed multigraph. Therefore, by Corollary~\ref{cor1}, it has a feedback vertex set $S$ of size at most
		\[
		|S| \le \frac{|V(D'')|}{2} = \frac{|V(D')|}{4}, 
		\]
		unless $D''\in \mathcal{C}_o$. 
		
		If $D'' \notin \mathcal{C}_o$, let $F \subseteq M$ be the set of arcs in $D'$ 
		corresponding to the vertices in $S$. We claim that $F$ is a feedback arc set 
		of $D'$. Indeed, let $C'$ be any cycle in $D'$, and let $C''$ be the 
		corresponding cycle in $D''$. Since $S$ is a feedback vertex set of $D''$, 
		we have $S \cap V(C'') \neq \emptyset$. Assume that $s \in S \cap V(C'')$ and it corresponds
		to an arc $xy \in M$. Then $C'$ must traverse $xy$, because $y$ is the 
		unique out-neighbour of $x$ and $x$ is the unique in-neighbour of $y$ 
		(as $x \in X^{-}$ and $y \in X^{+}$). Therefore, removing $xy$ destroys 
		the cycle $C'$. Since every cycle in $D'$ contains at least one arc from $F$, 
		it follows that $F$ is a feedback arc set of $D'$.
		
		 As $D'$ is degree-$3$, we have $a(D') = \frac{3|V(D')|}{2}$. Therefore,
\[
|F| = |S| = \frac{|V(D')|}{4} = \frac{a(D')}{6},
\]
a contradiction. Thus, we may assume that $D''\in \mathcal{C}_o$. \MAYZR{Note that in this case, $|M|=|V(D'')|=2k+1$ for some $k\geq 1$.}

Let $\mathcal{C}_o'$ denote the set of digraphs obtained from any digraph $H \in \mathcal{C}_o$ by splitting each vertex $v \in V(H)$ into two vertices $v^+$ and $v^-$, adding an arc from $v^-$ to $v^+$, and for each arc $uv \in A(H)$, adding an arc from $u^+$ to $v^-$.
\MAYZR{
\begin{figure}[h]
\tikzset{special/.style = {->, densely dotted, line width=0.03cm}}
\centering
\begin{subfigure}{0.45\textwidth}
\centering
\begin{tikzpicture}[scale=3.9]
\tikzset{original/.style = {->, line width=0.03cm}}
\node[vertexB] (x1) at (0,1) {$v_1^+$};
\node[vertexB] (x2) at (0,0.75) {$v_2^+$};
\node[vertexB] (x3) at (0,0.5) {$v_3^+$};
\node[vertexB] (x4) at (0,0.25) {$v_4^+$};
\node[vertexB] (x5) at (0,0) {$v_5^+$};

\node[vertexB] (y1) at (0.5,1) {$v_1^-$};
\node[vertexB] (y2) at (0.5,0.75) {$v_2^-$};
\node[vertexB] (y3) at (0.5,0.5) {$v_3^-$};
\node[vertexB] (y4) at (0.5,0.25) {$v_4^-$};
\node[vertexB] (y5) at (0.5,0) {$v_5^-$};
\draw [rounded corners=10pt] (-0.2,1.2) rectangle (0.2,-0.2);
\draw [rounded corners=10pt] (0.3,1.2) rectangle (0.7,-0.2);
\draw (0,1.3) node {$X^+$};
\draw (0.5,1.3) node {$X^-$};

\draw [original] (x1) to (y2);
\draw [original] (x2) to (y1);
\draw [original] (x2) to (y3);
\draw [original] (x3) to (y2);
\draw [original] (x3) to (y4);
\draw [original] (x4) to (y3);
\draw [original] (x4) to (y5);
\draw [original] (x5) to (y4);
\draw [original] (x5) to (y1);
\draw [original] (x1) to (y5);

\draw [special] (y1) to (x1);
\draw [special] (y2) to (x2);
\draw [special] (y3) to (x3);
\draw [special] (y4) to (x4);
\draw [special] (y5) to (x5);
\end{tikzpicture}
\subcaption{$D'$.}
\label{fig:c5}
\end{subfigure}
\begin{subfigure}{0.45\textwidth}
\centering
\begin{tikzpicture}[scale=2.1]
\tikzset{curve/.style = {->, bend left=20, line width=0.03cm}}
\foreach[count=\i] \lab in
{ $v_1$, $v_2$, $v_3$, $v_4$,
$v_5$ }
{
\node [vertexB] (v\i) at ({360/5*(\i-1)+18}:1)
{\lab};
}

\foreach \i in {1,...,4} {
\draw [curve] (v\i) to (v\the\numexpr\i+1\relax);
\draw [curve] (v\the\numexpr\i+1\relax) to (v\i);
}

\draw [curve] (v1) to (v5);
\draw [curve] (v5) to (v1);
\node at (0,-1.25) {};
\end{tikzpicture}    
\subcaption{$D''$.}
\label{fig:c5'}
\end{subfigure}
\caption{\MAYZR{$D'$ and $D''$, when $D''$ is obtained from an undirected 5-cycle by replacing each edge with a directed 2-cycle. The dotted arcs indicate the arcs in $M$.}}\label{fig:c5c5'}
\end{figure}
}

Observe that if $D'' \in \mathcal{C}_o$, then $D' \in \mathcal{C}_o'$
\MAYZR{(Fig. \ref{fig:c5c5'} illustrates how we get $D''$ from $D'$ when $D''$ is obtained from an undirected 5-cycle by replacing each edge with a directed 2-cycle)}. By Claim~E (i), $D$ can be obtained from $D'$ by subdividing arcs in $A(D') \setminus M$. Since $g(D) \ge 6$, we need to subdividing at least 2 times for every 4-cycles in $D'$. As  $D''$ has \MAYZR{$|V(D'')|=2k+1$ 2-cycles}, $D'$ has \MAYZR{$2k+1$} 4-cycles and therefore, \MAYZR{as $a(D')=3|M|=6k+3$,}
\MAYZR{ \[
a(D) \ge a(D') + 4k+2= 10k+5>6k+6,
\]
where the last inequality holds as $k\geq 1$.}
         Let $S$ be a minimum feedback vertex set of $D''$. Since $D'' \in \mathcal{C}_o$, we have
\[
|S| = \frac{|V(D'')| + 1}{2} = \MAYZR{k+1}.
\]
Let $M'$ be the set of all arcs in $M$ corresponding to the vertices in
$S$. Then, one can observe that $M'$ forms a feedback arc set of $D$.
Moreover, \MAYZR{as $a(D)>6k+6$,}
\[
|M'| = |S| = \MAYZR{k+1 = \frac{6k + 6}{6} < \frac{a(D)}{6}}.
\]
Hence, $M'$ is a desired feedback arc set of $D$, a contradiction. This completes the proof.	\end{proof}

	\section{Lower bounds on $\fas(\Delta,g)$ and upper bounds on $\fasd(\Delta,g)$}\label{sec:lub}

	
\GG{	Let $G$ be a directed (undirected, respectively) graph. For subsets $A,B \subseteq V(G)$, let $(A,B)$ be the bipartite subgraph of $G$ induced by the arcs with tail in $A$ and head in $B$ (induced by edges with one end-vertex in $A$ and another end-vertex in $B$, respectively). Let $a(A,B)$ ($e(A,B)$, respectively) be the number of arcs (edges, respectively) in $(A,B).$
}

    \FIX{If $G$ is a $d$-regular undirected graph, then the adjacency matrix
        of $G$ has $n$ real eigenvalues $\lambda_1 \geq \lambda_2 \geq
        \dots \geq \lambda_n$. Let $\lambda(G) = \max_{i=2}^n |\lambda_i|$ and
        let $\lambda'(G) = \max\{|\lambda_i| \mid |\lambda_i| \neq d, 1
    \leq i \leq n\}$.}
        A connected $d$-regular graph $G$ is a {\em Ramanujan graph} if
    $\lambda'(G)\leq 2{\sqrt {d-1}}$. For any two positive integers $p$ and
    $q$, let $\left(\frac{q}{p}\right)$ be the Legendre symbol; that is, it
    equals to $1$ if $q$ is a quadratic residue modulo $p$ and $-1$
    otherwise. Lubotzky, Phillps, and Sarnak \cite{LPS1988} gave explicit
    constructions of infinitely many $(p+1)$-regular Ramanujan graphs for
    every prime $p$ congruent to 1 mod 4.
    \begin{thm}\cite{LPS1988}\label{thm:1mod4} For every two unequal primes
        $5\leq p<q$ congruent to $1$ mod $4$, there is a $(p+1)$-regular
        Ramanujan graph $G$ with $n$ vertices and girth at least $\log_p
        n,$ where 
			
			$$n=\left\{
			\begin{array}{rcl}
				q(q^2-1)&  & if~\left(\frac{q}{p}\right)=-1;\\
				q(q^2-1)/2&  & if~\left(\frac{q}{p}\right)=1.
			\end{array}\right.$$
            \FIX{Furthermore, $G$ is bipartite iff
            $\left(\frac{q}{p}\right)=-1$.}
		\end{thm}
		


    The following lemma is known as the Expander Mixing Lemma \YZn{(see,
    e.g., \cite[Theorem 2.11]{KS2006}).}
	
	\begin{lemma}\label{lem:expandermixing}
		Let $G$ be a $d$-regular graph with order $n$ and $\lambda(G)\leq \lambda$. Then, for every two subsets $S, T\subseteq V(G)$
		\[\left|e(S,T)-\frac{d|S||T|}{n}\right|\leq
        \lambda\sqrt{|S||T|\left(1-\frac{|S|}{n}\right)\left(1-\frac{|T|}{n}\right)}.\]
	In particular, if $|S|=|T|=\frac{n}{2}$, we have that
	
	\begin{equation}\label{eq:eml}
		e(S,T)\geq \frac{(d-\lambda)n}{4}.
	\end{equation}
	\end{lemma}

    \FIX{
    Note that $\lambda(G)$ as used in Lemma \ref{lem:expandermixing} is
    different from $\lambda'(G)$ used in the definition of Ramanujan
    graphs. The distinction is important, since $\lambda(G) = d >
    \lambda'(G)$ if $G$ is a bipartite Ramanujan graph. However, if $G$
    is not bipartite, then $\lambda(G) = \lambda'(G)$. This follows from
    the fact that $d$ is a simple eigenvalue of the adjacency matrix of a
    connected graph $G$ (which is follows from the Perron-Frobenius
    theorem), and $\lambda_1 = -\lambda_n$ iff $G$ is bipartite (see e.g.
    \cite[Theorem 3.2.4]{CRS2009}).

    In order to obtain non-bipartite Ramanujan graphs from Theorem
    \ref{thm:1mod4}, we need the following.

    \begin{lemma}
        \label{lem:primesolution}
        For any odd prime power $p = r^d$ and any $k \geq 1$ there exists
        infinitely many prime solutions $x$ to the system of congruences 
        \begin{align*}
            x &\equiv 1 \pmod {2^k}\\
            x &\equiv 4 \pmod p.
        \end{align*}
        \end{lemma}

    \begin{proof}
        Note that $2^k$ and $p$ are coprime as $p$ is odd. Thus, by the
        Chinese Remainder Theorem, there exists a solution $x$
        and $x+2^kpn$ is a solution for any integer $n$. By Dirichlet's
        theorem on primes in arithmetic progressions, there exists
        infinitely many primes of the form $x+2^kpn$ so long as $x$ and
        $2^kp$ are coprime. We show that $x$ and $2^kp$ are coprime, which
        completes the proof.

        The only prime factors of $2^kp$ are $2$ and $r$. Since $x \equiv 1
        \pmod {2^k}$, $2$ does not divide $x$. Since $x \equiv 4 \pmod{r^d}$,
        there exists an integer $a$ such that $x = ar^d+4$. If $r$ divides
        $ar^d+4$, then $r$ divides $4$, contradicting that $r$ is odd.
    \end{proof}
    }
	
	Now we are ready to prove the following:
		\begin{thm}\label{thm:girthandlargefas}
            For any integer $g\geq 3$ and odd prime power $p \FIX{\ \equiv 1
            \pmod 4}$, there exists
            a $\frac{p+1}{2}$-regular orgraph $D$ with directed girth at
            least $g$ such that $\fas(D)\geq
            \frac{p+1-2\sqrt{p}}{4(p+1)}a(D)$ and therefore $\fasd(D)\leq
            \frac{4(p+1)}{p+1-2\sqrt{p}}$. 
		\end{thm}
	
	\begin{proof}
		Fix an arbitrary integer $g\geq 3$ \FIX{\ and odd prime power $p
        \equiv 1 \pmod 4$}. \FIX{By Lemma
        \ref{lem:primesolution}, there exists infinitely many primes $q$
        such that $q\equiv 1 \pmod{4}$ and $q\equiv 4 \pmod p$.} Then, by
        applying \FIX{Theorem \ref{thm:1mod4}}, we can choose $q$
        sufficiently large such that there is a \FIX{non-bipartite}
        $(p+1)$-regular Ramanujan graph $G$ with girth at least
        \FIX{$\log_p n \geq g$}. As $G$ is \FIX{a Eulerian} graph, it has a
            cycle decomposition, and therefore, we can obtain \FIX{a
            Eulerian} oriented graph $D$ from $G$ by orienting every cycle
        in the cycle decomposition as a directed one. 
		
        Note that by \FIX{Theorem \ref{thm:1mod4}}, \MAYZR{$n=\frac{q(q+1)(q-1)}{2}$ and therefore} $n$ is even. For an arbitrary
        ordering $\sigma$ of the vertex set $V(D)$, we denote by $A_\sigma$
        and $B_\sigma$ the set of the first and the last $n/2$ vertices,
        respectively. As $D$ is Eulerian, there are equally many arcs from
        $A_\sigma$ to $B_\sigma$ and from $B_\sigma$ to $A_\sigma$. Thus,
        by (\ref{eq:eml}), the number of backward arcs in $\sigma$ is
		
		\[\bas(D,\sigma)\geq a(B_\sigma,A_\sigma)= \frac{e(A_\sigma,B_\sigma)}{2}\geq \frac{p+1-2\sqrt{p}}{8}n=\frac{p+1-2\sqrt{p}}{4(p+1)}a(D). \]
		As the above inequality holds for any ordering, $\fas(D)\geq \frac{p+1-2\sqrt{p}}{4(p+1)}a(D)$ and therefore $\fasd(D)\leq a(D)/\fas(D)\leq \frac{4(p+1)}{p+1-2\sqrt{p}}$. This completes the proof.
	\end{proof}
	
    \FIX{Thus, we have the following corollaries as $5$ and $101$ are
    primes congruent to $1$ modulo $4$, $\frac{6-2\sqrt{5}}{24}>1/16$ and
$\frac{102-2\sqrt{101}}{408}>1/5$.}


		\begin{cor}\label{cor:d=4}
			For any integer $g\geq 3$, there exists a $3$-regular orgraph $D$ with directed girth at least $g$ such that $\fas(D)>a(D)/16 $. In particular, for every $g\geq 3$, $\fas(6,g)>\frac{1}{16}$ and $\fasd(6,g)\leq 15$.
		\end{cor}
		
		\begin{cor}
			For any integer $g\geq 3$, there exists a $51$-regular orgraph $D$ with directed girth at least $g$ such that $\fas(D)>a(D)/5$. In particular, for every $g\geq 3$, $\fas(102,g)>1/5$ and $\fasd(102,g)\leq 4$. 
		\end{cor}

\MAYZR{Note that Theorem
        \ref{thm:girthandlargefas} is applied to $101$ specifically since
        Theorem \ref{thm:girthandlargefas} cannot
        directly give an upper bound on $\fasd(\Delta,g)$ of less than $4$
    as $\frac{4(p+1)}{p+1-2\sqrt{p}} > 4$ for all $p > 1$ and $101$ is
the smallest prime $p \equiv 1 \pmod 4$ for which we get $\frac{4(p+1)}{p+1-2\sqrt{p}} < 5$.}

		The following result for maximum degree 3 can be obtained by applying a splitting operation to orgraphs obtained from Corollary \ref{cor:d=4}. 
		
        \begin{thm}\label{thm:d=3}
            For any integer $g\geq 3$, there exists an orgraph $D$ with
            $\Delta(D) = 3$, $\fas(D) > a(D)/95$ and directed girth at
            least $g$. In particular, for every $g\geq 3$,
            $\fas(3,g)>\frac{1}{95}$ and $\fasd(3,g)\leq 94$.
		\end{thm}

        \FIX{
\begin{proof}
    Fix $g \geq 3$. By Theorem \ref{thm:girthandlargefas} applied to $p=5$,
    let $D$ be a $3$-regular digraph with $g(D) \geq g$ and $\fas(D) \geq
    \frac{5+1-2\sqrt{5}}{4(5+1)}a(D)$.

    We now \textit{split} every vertex $v \in V(D)$ as follows.
    Let $v \in V(D)$ be a vertex and let $N^-(v) = \{w_1,w_2,w_3\}$ and
    $N^+(v) = \{q_1,q_2,q_3\}$. We replace $v$ by $4$ vertices $v'_s$,
    $v_s$, $v_t$, and $v'_t$ and the arcs $v'_sv_s$, $v_sv_t$, and
    $v_tv'_t$. We then add the arcs $w_1v_s,w_2v'_s,w_3v'_s$ and $v_tq_1,
    v'_tq_2, q_3v'_t$. See Figure \ref{fig:splitting}. We call the arcs
    $v'_sv_s,v_sv_t,v_tv'_t$ \textit{internal} arcs of $v$ and the arcs
    $w_1v_s,w_2v'_s,w_3v'_s$ and $v_tq_1,v'_tq_2,q_3v'_t$
    \textit{original}. Let $D'$ be the digraph obtained by splitting every
    vertex of $D$ in this way.

    First, we observe that every vertex in $D'$ has at most $3$ neighbors.
    Furthermore, $a(D') = a(D) + 3n(D) = 2a(D)$. We now show that $\fas(D')
    \geq (1/3) \fas(D)$. Suppose $F'$ is a feedback arc set of $D'$. We
    obtain a feedback arc set $F$ of $D$ with $|F| \leq 3|F'|$ as follows.
    For every original arc $xy$ in $F'$, simply replace $xy$ by the
    corresponding original arc in $D$. Now, for every $v \in V(D)$, replace
    any internal arc $xy$ of $v$ contained in $F'$, by the $3$ original
    arcs entering $\{v_s,v'_s\}$. Since every cycle which uses $xy$ must
    also use one of the original arcs entering $\{v_s,v'_s\}$, we obtain a
    feedback arc set $F$ of $D$. Since every arc in $F'$ was replaced by at
    most $3$ arcs, we have $|F| \leq 3|F'|$. Thus, $\fas(D') \geq
    (1/3)\fas(D)$.
    Now, \[
        \fas(D') \geq \frac{\fas(D)}{3} \geq
        \frac{5+1-2\sqrt{5}}{3 \cdot 4(5+1)}a(D) =
        \frac{5+1-2\sqrt{5}}{2\cdot 3 \cdot 4(5+1)}a(D') > \frac{1}{95}a(D')
    \]
    where we used in the third step that $a(D') = 2a(D)$.

    \begin{figure}[h]
        \centering
\begin{tikzpicture}[node distance=0.8cm]
    \node[vertexS] (v) at (0,0) {$v$};
    \node[vertexS] (w1) at (-2,1) {$w_1$};
    \node[vertexS] (w2) at (-2,0) {$w_2$};
    \node[vertexS] (w3) at (-2,-1) {$w_3$};

    \node[vertexS] (q1) at (2,1) {$q_1$};
    \node[vertexS] (q2) at (2,0) {$q_2$};
    \node[vertexS] (q3) at (2,-1) {$q_3$};

    \draw[arc] (w1) -- (v);
    \draw[arc] (w2) -- (v);
    \draw[arc] (w3) -- (v);

    \draw[arc] (v) -- (q1);
    \draw[arc] (v) -- (q2);
    \draw[arc] (v) -- (q3);

    \draw[arc] (2.75,0) to node[above] {\small Split $v$} (4.25,0);

    \begin{scope}[shift={(7,0)}]

        \node[vertexS] (vs) at (-0.5,0) {$v_s$};
        \node[vertexS] (vt) at (0.5,0) {$v_t$};

        \node[vertexS] (vsp) at (-1.25,-0.5) {$v'_s$};
        \node[vertexS] (vtp) at (1.25,-0.5) {$v'_t$};

        \node[vertexS] (w1) at (-2,1) {$w_1$};
        \node[vertexS] (w2) at (-2,0) {$w_2$};
        \node[vertexS] (w3) at (-2,-1) {$w_3$};

        \node[vertexS] (q1) at (2,1) {$q_1$};
        \node[vertexS] (q2) at (2,0) {$q_2$};
        \node[vertexS] (q3) at (2,-1) {$q_3$};

        \draw[arc] (w1) -- (vs);
        \draw[arc] (w2) -- (vsp);
        \draw[arc] (w3) -- (vsp);

        \draw[arc] (vt) -- (q1);
        \draw[arc] (vtp) -- (q2);
        \draw[arc] (vtp) -- (q3);

        \draw[arc] (vsp) -- (vs);
        \draw[arc] (vs) -- (vt);
        \draw[arc] (vt) -- (vtp);
    \end{scope}

\end{tikzpicture}
        \caption{Splitting a vertex $v$.}
        \label{fig:splitting}    
    \end{figure}

   \end{proof}
}	
		
				We now show that there is a constant $c>0$ such that $\fasd(\Delta,g)=2$ when $\Delta\geq c$ no matter how large $g$ is. To prove it we need the following special case of Hoeffding's inequality \YZn{\cite{Hoeffding1963}}.
			
			\begin{lemma}\label{lem: Hoeffding's ineq}
				Let $X_1$, $\dots$, $X_n$ be independent random variables such that $0\leq X_i\leq 1$ for all $i\in [n]$. Let $X=\sum_{i=1}^n X_i$. Then, for any real $\alpha>0$, we have that \[Pr[X-\mathbb{E}(X)\leq -\alpha]\leq e^{-2\alpha^2/n}.\]
			\end{lemma}
		\begin{thm}
			For any integer $g\geq 3$, there exists an orgraph $D$ with maximum degree $\Delta(D)=\MAYZR{1362}$, $\fas(D) > a(D)/3$ and directed girth at least $g$. In particular, for every $g\geq 3$, $\fas(\MAYZR{1362},g)>\frac{1}{3}$ and $\fasd(\MAYZR{1362},g)= 2$.
		\end{thm}
		\begin{proof}
            Fix $g \geq 3$. Let $d=\MAYZR{1362}$ and let $G$ be a \FIX{non-bipartite} $\MAYZR{1362}$-regular
            Ramanujan graph ($\lambda=\lambda(G)\leq 2\sqrt{d-1}$) with order
            $n\equiv 0 \pmod 8$ such that \FIX{$g(G) \geq g$}. \FIX{The existence
                of such graph is guaranteed by applying Theorem \ref{thm:1mod4}
                for $p=d-1=\MAYZR{1361}$ and a sufficiently large prime $q$ with
                $q\equiv 1 \pmod 8$ and $q \equiv 4 \pmod p$, where the existence of such
            prime $q$ follows from Lemma \ref{lem:primesolution}.}
		   
Let $S(G)$ be the set of all orderings on the vertex set $V(G)$. For any integer $1\leq i\leq 3$, we denote by $B(i)$ the set of all bit strings of length $i$. For any $\sigma\in S(G)$,  let $A_{\sigma}^0$ and $A_{\sigma}^1$ be the first $n/2$ vertices and the last $n/2$ vertices in the ordering $\sigma$, respectively. Similarly, for any $2\leq i\leq 3$ and bit string $\epsilon\in B(i-1)$, let $A^{\epsilon 0}_\sigma$ and $A_{\sigma}^{\epsilon1}$ be the first and second half of vertices in $A^{\epsilon}_{\sigma}$ in the ordering $\sigma$. Clearly, $|A_\sigma^0|=|A_\sigma^1|=n/2$ and generally for every integer $1\leq i\leq 3$ and $\epsilon\in B(i)$, since $n\equiv 0({\rm mod}~8)$ we have that

		   \begin{equation}\label{eq:size}
		   	|A^\epsilon_\sigma|=n/2^i.
		   \end{equation}
		    
		    \YZn{
		        \begin{figure}[h]
		 
		    	  		\begin{tikzpicture}[node distance=1.5cm]
		    		
		    		\node[vertexB] (p1) {$v_1$};
		    		\node[vertexB] (p2) [right of = p1] {$v_2$};
		    		\node[vertexB] (p3) [right of = p2] {$v_3$};
		    		\node[vertexB] (p4) [right of = p3] {$v_4$};
		    		\node[vertexB] (p5) [right of = p4] {$v_5$};
		    		\node[vertexB] (p6) [right of = p5] {$v_6$};
		    		\node[vertexB] (p7) [right of = p6] {$v_7$};
		    		\node[vertexB] (p8) [right of = p7] {$v_8$};
		    		
		    		\draw [thick] (1.075,0.2)--(1.375,0.2);
		    		\draw [thick] (1.075,-0.2)--(1.375,-0.2);
		    		\draw [thick] (1.075,0.2)--(1.375,-0.2);
		    		\draw [thick] (1.075,-0.2)--(1.375,0.2);
		    		
		    	    \draw [thick] (4.375,0.2)--(4.675,0.2);
		    	     \draw [thick] (4.375,-0.2)--(4.675,-0.2);
		    	      \draw [thick] (4.375,0.2)--(4.675,-0.2);
		    	       \draw [thick] (4.375,-0.2)--(4.675,0.2);

		    		\draw [rounded corners=10pt] (-0.3,-0.4) rectangle (1.1,0.4);
		    		\draw [rounded corners=10pt] (1.35,-0.4) rectangle (2.75,0.4);
		    		
		    		\draw [rounded corners=10pt] (3,-0.4) rectangle (4.4,0.4);
		    		\draw [rounded corners=10pt] (4.65,-0.4) rectangle (6.05,0.4);
                 	\draw (0.5,0.75) node {$A^{00}_\sigma$};
                 	\draw (2.1,0.75) node {$A^{01}_\sigma$};
                 	\draw (3.7,0.75) node {$A^{10}_\sigma$};
                 	\draw (5.3,0.75) node {$A^{11}_\sigma$};
		    		\draw (3,-1) node {(a) $Q^2_\sigma$ where $\sigma=(1,2,3,4,5,6,7,8)$.};
		    	\end{tikzpicture}\hfill
		    	\begin{tikzpicture}[node distance=1.5cm]
		    		
		    		\node[vertexB] (p1) {$v_2$};
		    		\node[vertexB] (p2) [right of = p1] {$v_1$};
		    		\node[vertexB] (p3) [right of = p2] {$v_4$};
		    		\node[vertexB] (p4) [right of = p3] {$v_3$};
		    		\node[vertexB] (p5) [right of = p4] {$v_6$};
		    		\node[vertexB] (p6) [right of = p5] {$v_5$};
		    		\node[vertexB] (p7) [right of = p6] {$v_8$};
		    		\node[vertexB] (p8) [right of = p7] {$p_7$};
		    		    		\draw [thick] (1.075,0.2)--(1.375,0.2);
		    		\draw [thick] (1.075,-0.2)--(1.375,-0.2);
		    		\draw [thick] (1.075,0.2)--(1.375,-0.2);
		    		\draw [thick] (1.075,-0.2)--(1.375,0.2);
		    		
		    		\draw [thick] (4.375,0.2)--(4.675,0.2);
		    		\draw [thick] (4.375,-0.2)--(4.675,-0.2);
		    		\draw [thick] (4.375,0.2)--(4.675,-0.2);
		    		\draw [thick] (4.375,-0.2)--(4.675,0.2);
		    		  		\draw [rounded corners=10pt] (-0.3,-0.4) rectangle (1.1,0.4);
		    		\draw [rounded corners=10pt] (1.35,-0.4) rectangle (2.75,0.4);
		    		
		    		\draw [rounded corners=10pt] (3,-0.4) rectangle (4.4,0.4);
		    		\draw [rounded corners=10pt] (4.65,-0.4) rectangle (6.05,0.4);
		    		             	\draw (0.5,0.75) node {$A^{00}_{\sigma'}$};
		    		\draw (2.1,0.75) node {$A^{01}_{\sigma'}$};
		    		\draw (3.7,0.75) node {$A^{10}_{\sigma'}$};
		    		\draw (5.3,0.75) node {$A^{11}_{\sigma'}$};
		    			\draw (3,-1) node {(b) $Q^2_{\sigma'}$ where $\sigma'=(2,1,4,3,6,5,8,7)$.};
		    	\end{tikzpicture}
		    
		    		  \vspace{0.4cm}
		    		  \begin{subfigure}{\textwidth}
		    		  			    		\centering
		    	\begin{tikzpicture}[node distance=1.5cm]

		    		\node[vertexB] (p1) {$v_5$};
		    		\node[vertexB] (p2) [right of = p1] {$v_6$};
		    		\node[vertexB] (p3) [right of = p2] {$v_7$};
		    		\node[vertexB] (p4) [right of = p3] {$v_8$};
		    		\node[vertexB] (p5) [right of = p4] {$v_1$};
		    		\node[vertexB] (p6) [right of = p5] {$v_2$};
		    		\node[vertexB] (p7) [right of = p6] {$v_3$};
		    		\node[vertexB] (p8) [right of = p7] {$v_4$};
		    		    		\draw [thick] (1.075,0.2)--(1.375,0.2);
		    		\draw [thick] (1.075,-0.2)--(1.375,-0.2);
		    		\draw [thick] (1.075,0.2)--(1.375,-0.2);
		    		\draw [thick] (1.075,-0.2)--(1.375,0.2);
		    		
		    		\draw [thick] (4.375,0.2)--(4.675,0.2);
		    		\draw [thick] (4.375,-0.2)--(4.675,-0.2);
		    		\draw [thick] (4.375,0.2)--(4.675,-0.2);
		    		\draw [thick] (4.375,-0.2)--(4.675,0.2);
		    		  		\draw [rounded corners=10pt] (-0.3,-0.4) rectangle (1.1,0.4);
		    		\draw [rounded corners=10pt] (1.35,-0.4) rectangle (2.75,0.4);
		    		
		    		\draw [rounded corners=10pt] (3,-0.4) rectangle (4.4,0.4);
		    		\draw [rounded corners=10pt] (4.65,-0.4) rectangle (6.05,0.4);
		    		   \draw (0.5,0.75) node {$A^{00}_{\sigma''}$};
		    		\draw (2.1,0.75) node {$A^{01}_{\sigma''}$};
		    		\draw (3.7,0.75) node {$A^{10}_{\sigma''}$};
		    		\draw (5.3,0.75) node {$A^{11}_{\sigma''}$};
		    			\draw (3,-1) node {(c) $Q^2_{\sigma''}$ where $\sigma''=(5,6,7,8,1,2,3,4)$.};
		    	\end{tikzpicture}
		    	  \end{subfigure}

		    	\caption{\YZn{Examples of equivalent and non-equivalent graphs, where $Q$ is an arbitrary graph on eight vertices. One can check from the definition that $Q^2_{\sigma}$ and $Q^2_{\sigma'}$ are equivlent w.r.t. their orderings, and that $Q^2_{\sigma}$ and $Q^2_{\sigma''}$ are non-equivalent w.r.t. their orderings. }}
		    	\label{identical}
		    \end{figure}
		}
		   Let $G_\sigma^1$ be the bipartite subgraph $(A_\sigma^0, A_\sigma^1)$ of $G$, and for every integer $2\leq i\leq 3$, let $G_\sigma^i=\cup_{\epsilon\in B(i-1)}(A^{\epsilon0}_\sigma,A^{\epsilon1}_{\sigma})$. For every $1\leq i\leq 3$ and two different orderings $\sigma$ and $\sigma'$, we say that $G_\sigma^i$ and $G_{\sigma'}^i$ are {\em \YZn{equivalent w.r.t. their orderings}} if and only if $A^\epsilon_\sigma=A^\epsilon_{\sigma'}$ for all bit strings $\epsilon\in B(i)$. Note that whether $G_\sigma^i$ and $G_{\sigma'}^i$ are \YZn{equivalent w.r.t. their orderings} only depends on the sets $A^\epsilon_\sigma$s and $A^\epsilon_{\sigma'}$s. $G_\sigma^i$ and $G_{\sigma'}^i$ can be \YZn{equivalent w.r.t. their orderings} when $\sigma\neq \sigma'$ (see, e.g., Fig. \ref{identical}(a) and \ref{identical}(b)). Clearly, if $G_\sigma^i$ and $G_{\sigma'}^i$ are \YZn{equivalent w.r.t. their orderings} then $G_\sigma^i=G_{\sigma'}^i$. In contrast, $G_\sigma^i$ and $G_{\sigma'}^i$ can be \YZn{non-equivalent w.r.t. their orderings} even if $G_\sigma^i=G_{\sigma'}^i$ (see, e.g., Fig. \ref{identical}(a) and \ref{identical}(c)). \MAYZR{Let $\Omega_i$ be a set containing one arbitrary representative from each equivalence class in $\{G^i_\sigma : \sigma \in S(G)\}$.
		   }. Then for every $H\in \Omega_i$ where $1\leq i\leq 3$, by Lemma \ref{lem:expandermixing} and (\ref{eq:size}), we have that 
		   
		   \YZn{
		   \begin{eqnarray*}
		   \left|e(H)-2^{i-1}\cdot\frac{d\cdot(n/2^i)^2}{n}\right|\leq 2^{i-1} \cdot\lambda\sqrt{ (n/2^i)^2\cdot(1-1/2^i)^2},
		   \end{eqnarray*}
		  which can be rewritten to}
		   
		   \begin{equation}\label{eq:m1}
		   	 \frac{(d-(2^i-1)\lambda)n}{2^{i+1}}\leq e(H)\leq \frac{(d+(2^i-1)\lambda)n}{2^{i+1}}.
		   \end{equation}
		   
		  \YZn{ 
 Let us analyze the cardinality of $\Omega_i$. 

\2
		   
		   {\bf Claim~A:}  For every $1\leq i\leq 3$, $|\Omega_i|\leq 2^{i(n-1)}$.}
		\begin{proof}[Proof of Claim A]\renewcommand{\qedsymbol}{$\diamond$}
				Note that if $n$ can be divided by $2^i$, then \MAYZR{there are \(2^{i\cdot n}\) ways to assign the \(n\) vertices to \(2^i\) labeled parts (empty parts allowed)}, and at least $(2^i-1)2^{i(n-1)}$ of them are not {\em equitable partitions} (i.e., partitions where each set has size exactly $n/2^i$) since one can fix an element and \MAYZR{assign} the rest $n-1$ vertices into $2^i$ labeled parts (\MAYZR{again empty parts allowed and therefore} there are $2^{i(n-1)}$ ways of doing it) and there are at least $2^i-1$ ways of placing the fixed vertex such that the resulting partition is not equitable \YZn{and one can observe that all partitions obtained in this way are different from each other}. Thus, there are at most $2^{i\cdot n}-(2^i-1)2^{i(n-1)}=2^{i(n-1)}$ ways of partitioning an $n$-set into $2^i$ sets with size $n/2^i$. Therefore, as every element in $\Omega_i$ is uniquely defined by a labeled equitable partition (labeled by the set $B(i)$) into $2^i$ sets, for every $1\leq i\leq 3$, we have that 	$|\Omega_i|\leq 2^{i(n-1)}$.
		\end{proof}
		
	   Now, let $\mathbb{D}$ be the \YZn{probability space of orgraphs} obtained from $G$ by assigning one of the two directions for every edge of $G$ independently and uniformly (with probability 1/2). For any $1\leq i\leq 3$ and ordering $\sigma\in S(G)$, let $\alpha_i(G_\sigma^i)=\sqrt{\ln(2)e(G_\sigma^i)ni/2}$ and let $X_{G_\sigma^i}$ be the random variable such that $X_{G_\sigma^i}(D)=\sum_{\epsilon\in B(i-1)}a_D(A^{\epsilon1}_\sigma, A^{\epsilon0}_\sigma)$, for all outcomes $D$ of $\mathbb{D}$. Then, by Lemma \ref{lem: Hoeffding's ineq}, for every $1\leq i\leq 3$ and $H\in \Omega_i$ we have that 
		\[Pr[X_H-e(H)/2\leq -\alpha_i(H)]\leq e^{-2\alpha_i(H)^2/e(H)}=2^{-i\cdot n}.\]
		
		Thus, by Claim A,
		\[\sum_{i=1}^3\sum_{H\in\Omega_i} Pr[X_H-e(H)/2\leq -\alpha_i(H)]\leq \sum_{i=1}^3 2^{i(n-1)}\cdot 2^{-i\cdot n}=\frac{7}{8}<1,\]
		which implies that there is an orientation $D$ of $G$ such that
		
		\begin{equation}\label{eq:X_H}
			X_H(D)\geq \frac{e(H)}{2}-\alpha_i(H).
		\end{equation}
		for all $i\in \{1,2,3\}$ and $H\in \Omega_i$. Fix this orientation $D$. By the definition of $\Omega_i$, for any fixed ordering $\sigma$ and $i\in \{1,2,3\}$ there is an element $H_i\in \Omega_i$ such that $H_i$ and $G_\sigma^i$ are \YZn{equivalent w.r.t. their orderings} and in particular $X_{H_i}(D)=X_{G_\sigma^i}(D)$. Thus, \GG{the number of backward arcs in $\sigma$}
		 
		\begin{eqnarray*}
			\bas(D,\sigma)&\geq& \sum_{i=1}^3X_{G_\sigma^i}(D)=\sum_{i=1}^3X_{H_i}(D)\\
			&\geq& \sum_{i=1}^3\left(\frac{e(H_i)}{2}-\alpha_i(H_i)\right)=\sum_{i=1}^3\left(\frac{e(H_i)}{2}-\sqrt{\frac{\ln(2)e(H_i)ni}{2}}\right)\\
		&\geq & \frac{(7d-17\lambda)n}{32}-\sqrt{\frac{\ln(2)}{2}}\YZn{\left(\sqrt{\frac{d-\lambda}{4}}+\sqrt{\frac{d-3\lambda}{4}}+\sqrt{\frac{3d-21\lambda}{16}}\right)n}\\
			&>& \frac{(7d-34\sqrt{d})n}{32}-\MAYZR{\sqrt{\frac{\ln(2)}{2}}\left(\frac{4+\sqrt{3}}{4}\sqrt{d}\right)}n\\
			&=& \frac{7dn}{32}-\MAYZR{\left(\frac{17}{16}+\frac{4+\sqrt{3}}{4}\sqrt{\frac{\ln(2)}{2}}\right)}\sqrt{d}n\\
			&> & \left(\frac{7}{16}-\frac{17}{8\sqrt{d}}-\MAYZR{\frac{17}{10\sqrt{d}}}\right)a(D)\\
			&>& \frac{a(D)}{3}, 
		\end{eqnarray*}
		where the second inequality follows from (\ref{eq:X_H}); the third inequality follows from (\ref{eq:m1}) \YZn{and the fact that the derivative of the function $f(x)=\frac{x}{2}-\sqrt{\frac{\ln(2)xni}{2}}$ is positive when $x>  \frac{\ln(2)in}{2}$ and 
		\[ \frac{\ln(2)in}{2}\leq \frac{\ln(2)3n}{2}<\frac{(d-(8-1)\cdot 2\sqrt{d})n}{16}\leq  \frac{(d-(2^i-1)\lambda)n}{2^{i+1}},\]
		as $1\leq i\leq 3$; }\MAYZR{the fourth inequality follows from the fact that $0\leq \lambda< 2\sqrt{d}$}, and the rest of the inequalities holds as $d=\MAYZR{1362}$ and $\MAYZR{\frac{4+\sqrt{3}}{2}\sqrt{\frac{\ln(2)}{2}} < 1.7}$. As the above inequalites are true for all ordering $\sigma$, $\fas(D)>\frac{a(D)}{3}$. In addition, by Theorem \ref{thm:1mod4} and the fact that $n\geq p^g$, $D$ has girth at least $\log_{p}(n)\geq g$. This completes the proof.
		\end{proof}
		\YZn{
		\begin{remarks}
			By replacing $d=\MAYZR{1362}$ with $d=\MAYZR{422}$ in the above proof, one can show that for every integer $g\geq 3$, $\fas(\MAYZR{422},g)>\frac{1}{4}$ and therefore $\fasd(\MAYZR{422}, g)\leq 3$.  
		\end{remarks}
	}
Given that we clearly have $\fasd(\Delta,g) \leq g$, it is natural to ask
for which values of $\Delta$ and $g$ this bound can actually be achieved.
In this paper, we have shown that $\fasd(3,g) = g$ for \YZ{$g=3,4,5$ and
$\fasd(4,3) = 3$}. Furthermore, Theorem \ref{thm:d=3} implies
that $\fasd(3,g) < g$ for $g \geq \MAYZR{95}$ and Corollary \ref{cor:d=4} implies
that $\fasd(4,g) < g$ for $g \geq 16$. \MAYZR{A natural problem is to
determine the smallest value of $g$ for which $\fasd(3,g)<g$ (and similarly
when 3 is replaced by a larger integer). We now show that 
$\fasd(4,g) < g$ for $g=6$, $\fasd(5,g) < g$ for $g=4$, and $\fasd(3,g) <
g$ for $g=9$. Afterwards, in Theorem~\ref{lemB}, we give a
general upper bound for $\fasd(3,g)$ which, in combination with $\fasd(3,9)
< 9$, allows us to conclude that $\fasd(3,g) < g$ for all $g \geq 8$}.

\begin{thm}\label{thm:small cases}
    We have $\fasd(5,4) < 4$, $\fasd(4,6) < 6$, \MAYZR{and $\fasd(3,9) < 9$.}
\end{thm}

\begin{proof}
    To show $\fasd(\Delta,g) < g$, it suffices to exhibit \YZ{an orgraph} $D$
    with $\Delta(D) \leq \Delta$ and $g(D) \geq g$ such that $\fasd(D) <
    g$. We think of a decomposition of $D$ into $g$ feedback arc sets as a
    coloring of the arcs of $D$ with $g$ colors such that all cycles
    contain all $g$ colors.

    We start by showing $\fasd(5,4) < 4$. Let $H_5$ be obtained from
    $K_{5,5} = (X \cup Y,E)$ by orienting a matching $M$ from $X$ to $Y$
    and all other arcs from $Y$ to $X$. Then $\Delta(H_5) = 5$ and $g(H_5)
    = 4$, see Figure~\ref{fig:D5}. Now, for any pair of
    distinct arcs $a,b \in M$, there is a $4$-cycle in $H_5$ containing
    both $a$ and $b$. If $a$ and $b$ lie on a common $4$-cycle $C$, then
    they cannot have the same color since then $C$ would contain at most
    $3$ distinct colors. Thus, all $5$ arcs in $M$ must have pairwise
    distinct colors, which is not possible using $4$ colors.

    We now show that $\fasd(4,6) < 6$. Let $H_4$ be the digraph obtained by
    taking a 3-regular tournament $T$ on $7$ vertices $x_1,x_2,\dots,x_7$ and
    \textit{splitting} every vertex $x_i$ to two vertices $y_i$ and $z_i$.
    By splitting we mean replacing $x_i$ by $y_i$ and $z_i$ and the arc
    $y_iz_i$ such that all in-neighbors (out-neighbors) of $x_i$ become
    (out-neighbors) in-neighbors of $y_i$ ($z_i$), see Figure \ref{fig:D4}.
    We have $\Delta(\YZ{H_4}) = 4$ and $g(\YZ{H_4}) = 2g(T) = 6$. Using the pigeonhole
    principle, one can easily show that for any pair $x_i,x_j$ of vertices
    in $T$, there exists a $3$-cycle containing $x_i$ and $x_j$. It follows
    that there is a $6$-cycle containing $y_iz_i$ and $y_jz_j$ for any pair
    of indices $i,j \in [7]$. Thus, the $7$ arcs
    $y_1z_1,y_2z_2,\dots,y_7z_7$ must have pairwise distinct colors, which
    is not possible using $6$ colors.

    \GGR{ Lastly, we show that $\fasd(3,9)<9$. Let $H_3$ be the digraph
        obtained by taking a 2-regular tournament $T$ on 5 vertices $x_1, x_2,
        \dots, x_5$ and splitting every vertex $x_i$ to three vertices $a_i$,
        $b_i$, and $c_i$. Here by splitting we mean replacing $x_i$ by $a_i$,
        $b_i$, and $c_i$ and the arcs $a_{i}b_{i}$, $b_{i}c_{i}$ such that all
        in-neighbors (out-neighbors) of $x_i$ become in-neighbors (out-neighbors)
        of $a_i$ ($c_i$). We have $\Delta(H_3)=3$ and $g(H_3)=3g(T)=9$.
        \MAYZR{See Figure~\ref{fig:D3}}. Using the
        pigeonhole principle, one can easily show that for any pair $x_i$, $x_j$ of
        vertices in $T$, there exists a 3-cycle containing $x_i$ and $x_j$. It
        follows that every pair of the 10 arcs $a_{i}b_{i}$, $b_{i}c_{i}$, $i\in
        [5]$ in $H_3$ lie on a common 9-cycle and thus must have pairwise distinct
    colors, which is not possible with 9 colors.}
\end{proof}

\begin{figure}[h]
    \tikzset{special/.style = {->, densely dotted, line width=0.03cm}}
    \centering
    \begin{subfigure}{0.23\textwidth}
            \centering
        \begin{tikzpicture}[scale=3.9]
            \node[vertexB] (x1) at (0,1) {$x_1$};
            \node[vertexB] (x2) at (0,0.75) {$x_2$};
            \node[vertexB] (x3) at (0,0.5) {$x_3$};
            \node[vertexB] (x4) at (0,0.25) {$x_4$};
            \node[vertexB] (x5) at (0,0) {$x_5$};

            \node[vertexB] (y1) at (0.5,1) {$y_1$};
            \node[vertexB] (y2) at (0.5,0.75) {$y_2$};
            \node[vertexB] (y3) at (0.5,0.5) {$y_3$};
            \node[vertexB] (y4) at (0.5,0.25) {$y_4$};
            \node[vertexB] (y5) at (0.5,0) {$y_5$};

            \draw [->, line width=0.03cm, color=gray, draw opacity=0.6] (y1) to (x2);
            \draw [->, line width=0.03cm, color=gray, draw opacity=0.6] (y1) to (x3);
            \draw [->, line width=0.03cm, color=gray, draw opacity=0.6] (y1) to (x4);
            \draw [->, line width=0.03cm, color=gray, draw opacity=0.6] (y1) to (x5);

            \draw [->, line width=0.03cm, color=gray, draw opacity=0.6] (y2) to (x1);
            \draw [->, line width=0.03cm, color=gray, draw opacity=0.6] (y2) to (x3);
            \draw [->, line width=0.03cm, color=gray, draw opacity=0.6] (y2) to (x4);
            \draw [->, line width=0.03cm, color=gray, draw opacity=0.6] (y2) to (x5);

            \draw [->, line width=0.03cm, color=gray, draw opacity=0.6] (y3) to (x1);
            \draw [->, line width=0.03cm, color=gray, draw opacity=0.6] (y3) to (x2);
            \draw [->, line width=0.03cm, color=gray, draw opacity=0.6] (y3) to (x4);
            \draw [->, line width=0.03cm, color=gray, draw opacity=0.6] (y3) to (x5);

            \draw [->, line width=0.03cm, color=gray, draw opacity=0.6] (y4) to (x1);
            \draw [->, line width=0.03cm, color=gray, draw opacity=0.6] (y4) to (x2);
            \draw [->, line width=0.03cm, color=gray, draw opacity=0.6] (y4) to (x3);
            \draw [->, line width=0.03cm, color=gray, draw opacity=0.6] (y4) to (x5);

            \draw [->, line width=0.03cm, color=gray, draw opacity=0.6] (y5) to (x1);
            \draw [->, line width=0.03cm, color=gray, draw opacity=0.6] (y5) to (x2);
            \draw [->, line width=0.03cm, color=gray, draw opacity=0.6] (y5) to (x3);
            \draw [->, line width=0.03cm, color=gray, draw opacity=0.6] (y5) to (x4);

            \draw [->, line width=0.03cm] (x1) to (y1); 
            \draw [->, line width=0.03cm] (x2) to (y2); 
            \draw [->, line width=0.03cm] (x3) to (y3); 
            \draw [->, line width=0.03cm] (x4) to (y4); 
            \draw [->, line width=0.03cm] (x5) to (y5);

        \end{tikzpicture}
        \subcaption{The digraph $H_5$.}
        \label{fig:D5}
    \end{subfigure}
\begin{subfigure}{0.29\textwidth}
    \centering
    \begin{tikzpicture}[scale=2]
        \tikzset{original/.style = {->, line width=0.03cm}}
        \node[vertexB] (a1) at (1.0, 0.0) {$y_1$};
        \node[vertexB] (b1) at (0.901, 0.4339) {$z_1$};
        \node[vertexB] (a2) at (0.6235, 0.7818) {$y_2$};
        \node[vertexB] (b2) at (0.2226, 0.9749) {$z_2$};
        \node[vertexB] (a3) at (-0.2225, 0.9749) {$y_3$};
        \node[vertexB] (b3) at (-0.6234, 0.7819) {$z_3$};
        \node[vertexB] (a4) at (-0.9009, 0.434) {$y_4$};
        \node[vertexB] (b4) at (-1.0, 0.0001) {$z_4$};
        \node[vertexB] (a5) at (-0.901, -0.4338) {$y_5$};
        \node[vertexB] (b5) at (-0.6236, -0.7818) {$z_5$};
        \node[vertexB] (a6) at (-0.2226, -0.9749) {$y_6$};
        \node[vertexB] (b6) at (0.2224, -0.975) {$z_6$};
        \node[vertexB] (a7) at (0.6234, -0.7819) {$y_7$};
        \node[vertexB] (b7) at (0.9009, -0.434) {$z_7$};

        \draw[special] (a1) -- (b1);
        \draw[special] (a2) -- (b2);
        \draw[special] (a3) -- (b3);
        \draw[special] (a4) -- (b4);
        \draw[special] (a5) -- (b5);
        \draw[special] (a6) -- (b6);
        \draw[special] (a7) -- (b7);

        \draw[original] (b1) -- (a2);
        \draw[original] (b1) -- (a3);
        \draw[original] (b1) -- (a4);

        \draw[original] (b2) -- (a3);
        \draw[original] (b2) -- (a4);
        \draw[original] (b2) -- (a5);

        \draw[original] (b3) -- (a4);
        \draw[original] (b3) -- (a5);
        \draw[original] (b3) -- (a6);

        \draw[original] (b4) -- (a5);
        \draw[original] (b4) -- (a6);
        \draw[original] (b4) -- (a7);

        \draw[original] (b5) -- (a6);
        \draw[original] (b5) -- (a7);
        \draw[original] (b5) -- (a1);

        \draw[original] (b6) -- (a7);
        \draw[original] (b6) -- (a1);
        \draw[original] (b6) -- (a2);

        \draw[original] (b7) -- (a1);
        \draw[original] (b7) -- (a2);
        \draw[original] (b7) -- (a3);
    \end{tikzpicture}    
    \subcaption{The digraph $H_4$.}
    \label{fig:D4}
    \end{subfigure}
\MAYZR{\begin{subfigure}{0.27\textwidth}
    \centering
\begin{tikzpicture}[scale=1.9]
\tikzset{original/.style = {->, line width=0.03cm}}

\node[vertexB] (a1) at (1.0000,0.0000) {$a_1$};
\node[vertexB] (b1) at (0.9135,0.4067) {$b_1$};
\node[vertexB] (c1) at (0.6691,0.7431) {$c_1$};

\node[vertexB] (a2) at (0.3090,0.9511) {$a_2$};
\node[vertexB] (b2) at (-0.1045,0.9945) {$b_2$};
\node[vertexB] (c2) at (-0.5000,0.8660) {$c_2$};

\node[vertexB] (a3) at (-0.8090,0.5878) {$a_3$};
\node[vertexB] (b3) at (-0.9781,0.2079) {$b_3$};
\node[vertexB] (c3) at (-0.9781,-0.2079) {$c_3$};

\node[vertexB] (a4) at (-0.8090,-0.5878) {$a_4$};
\node[vertexB] (b4) at (-0.5000,-0.8660) {$b_4$};
\node[vertexB] (c4) at (-0.1045,-0.9945) {$c_4$};

\node[vertexB] (a5) at (0.3090,-0.9511) {$a_5$};
\node[vertexB] (b5) at (0.6691,-0.7431) {$b_5$};
\node[vertexB] (c5) at (0.9135,-0.4067) {$c_5$};

\draw[special] (a1)--(b1); \draw[special] (b1)--(c1);
\draw[special] (a2)--(b2); \draw[special] (b2)--(c2);
\draw[special] (a3)--(b3); \draw[special] (b3)--(c3);
\draw[special] (a4)--(b4); \draw[special] (b4)--(c4);
\draw[special] (a5)--(b5); \draw[special] (b5)--(c5);

\draw[original] (c1)--(a2); \draw[original] (c1)--(a3);
\draw[original] (c2)--(a3); \draw[original] (c2)--(a4);
\draw[original] (c3)--(a4); \draw[original] (c3)--(a5);
\draw[original] (c4)--(a5); \draw[original] (c4)--(a1);
\draw[original] (c5)--(a1); \draw[original] (c5)--(a2);

\end{tikzpicture}
\subcaption{The digraph $H_3$.}
\label{fig:D3}
\end{subfigure}}
 
    \caption{The digraphs $H_5$, $H_4$, \MAYZR{and $H_3$ from the proof of Theorem
        \ref{thm:small cases}}. The arcs resulting
    from splitting a vertex are drawn with dotted lines.}
\end{figure}

\begin{thm} \label{lemB}
    For all even $g \geq 4$, we have $\fasd(3,g) \leq g - \left\lfloor \frac{g}{4} -1 \right\rfloor$.
For all odd $g \geq 7$, we have $\fasd(3,g) \leq g - \left\lfloor \frac{g-7}{4} \right\rfloor$.
\MAYZR{As also $\fasd(3,9) < 9$, we have $\fasd(3,g)<g$ for all $g\ge 8.$}
\end{thm}
\begin{proof}
  We first consider the even $g$ case, so let
$g \geq 4$ be even where $g=2k$. 
We will define the digraph $D_g$ as follows.
Let $P_1$, $P_2$ and $P_3$ be three vertex disjoint paths of length $k-1$,
where $P_j=p_1^j p_2^j \ldots p_k^j$ for all $j \in [3]$.  Now add all arcs from $p_k^i$ to $p_1^j$ for all $i,j \in [3]$, where $i \not= j$.
We have now constructed $D_g$ (i.e. $D_{2k}$). See Figure~\ref{D-digraph} for an illustration of $D_8$.
Note that the arc set $A(P_i) \cup A(P_j)$ belongs to a $g$-cycle in $D_g$ for all $1 \leq i < j \leq 3$.

\GGR{Assume that there is a good $c$-arc-coloring of $D_g$ (recall that each color corresponds to an FAS).}
Consider the three $g$-cycles $C_{12} = P_1 P_2 p_1^1$, $C_{13} = P_1 P_3 p_1^1$ and $C_{23} = P_2 P_3 p_1^2$.
Note that all arcs in $A'=A(P_1) \cup A(P_2) \cup A(P_3)$ belong to exactly two of the cycles $C_{12}$, $C_{13}$ and $C_{23}$
and all arcs in $A'' = \{p_k^1 p_1^2, p_k^2 p_1^1, p_k^1 p_1^3, p_k^3 p_1^1, p_k^2 p_1^3, p_k^3 p_1^2 \}$ belong to exactly one of the cycles. 
So, every color must appear on at least two of the arcs in $A' \cup A''$, which implies the following:
\[
c \leq \frac{|A' \cup A''|}{2} = \frac{3(k-1)+6}{2} = \frac{3g}{4} + \frac{3}{2} 
\]
Noting that $c$ is an integer and by 
considering the case when $g$ is
divisible by four and the case when it is not, we obtain the following, completing the proof of the even case.
\[
c \leq \left\lfloor \frac{3g}{4} + \frac{3}{2} \right\rfloor =  g - \left\lfloor \frac{g}{4} -1 \right\rfloor 
\]
Now consider the case when $g$ is odd and note that the following holds.
\[
\fasd(3,g) \leq \fasd(3,g+1) \leq g+1 - \left\lfloor \frac{g+1}{4} -1 \right\rfloor = g - \left\lfloor \frac{g-7}{4} \right\rfloor
\]
\MAYZR{This completes the proof.}
\end{proof}

\begin{figure}[htb]
\begin{center}
\tikzstyle{vertexX}=[circle, draw, top color=white, bottom color=white, minimum size=8pt, scale=0.65, inner sep=0.8pt]
\tikzstyle{vertexU}=[circle, draw, top color=white, bottom color=white, minimum size=6pt, scale=0.4, inner sep=0.8pt]

\begin{tikzpicture}[scale=0.3]
  \node (p11) at (0,9) [vertexX]{$p_1^1$}; 
  \node (p12) at (5,9) [vertexX]{$p_2^1$};
  \node (p13) at (10,9) [vertexX]{$p_3^1$};
  \node (p14) at (15,9) [vertexX]{$p_4^1$};

  \node (p21) at (0,5) [vertexX]{$p_1^2$};
  \node (p22) at (5,5) [vertexX]{$p_2^2$};
  \node (p23) at (10,5) [vertexX]{$p_3^2$};
  \node (p24) at (15,5) [vertexX]{$p_4^2$};

  \node (p31) at (0,1) [vertexX]{$p_1^3$};
  \node (p32) at (5,1) [vertexX]{$p_2^3$};
  \node (p33) at (10,1) [vertexX]{$p_3^3$};
  \node (p34) at (15,1) [vertexX]{$p_4^3$};

  \draw[->, line width=0.03cm] (p11) -- (p12);
  \draw[->, line width=0.03cm] (p12) -- (p13);
  \draw[->, line width=0.03cm] (p13) -- (p14);

  \draw[->, line width=0.03cm] (p21) -- (p22);
  \draw[->, line width=0.03cm] (p22) -- (p23);
  \draw[->, line width=0.03cm] (p23) -- (p24);

  \draw[->, line width=0.03cm] (p31) -- (p32);
  \draw[->, line width=0.03cm] (p32) -- (p33);
  \draw[->, line width=0.03cm] (p33) -- (p34);

  \draw[->, line width=0.03cm] (p14) -- (p21);
  \draw[->, line width=0.03cm] (p14) -- (p31);
  \draw[->, line width=0.03cm] (p24) -- (p11);
  \draw[->, line width=0.03cm] (p24) -- (p31);
  \draw[->, line width=0.03cm] (p34) -- (p11);
  \draw[->, line width=0.03cm] (p34) -- (p21);

\end{tikzpicture} 
\caption{The digraph \MAYZR{$D_g$ for $g=8$ from the proof of Theorem \ref{lemB}.}}
\label{D-digraph}
\end{center}
\end{figure}

\YZn{
\section{Conclusion}\label{sec:conclusion}
In this paper, we introduced a new parameter $\fasd(D)$ for orgraphs $D$,
through which we prove several results for the feedback arc sets of
weighted orgraphs with bounded maximum degree and girth. In particular, we
have shown that $\fasd(4,3)=3$, $\fasd(3,g)=g$ for all $g\in\{3,4,5\}$ and
that $\fasd(\Delta, g)$ is finite and bounded from above by \FIX{$94$} for
all $\Delta\geq 3$ and $g\geq 3$. We also obtained some better upper bound
for $\fasd(\Delta,g)$ when $\Delta$ is large, and especially we showed that
$\fasd(\Delta,g)=2$ for all $\Delta\geq \MAYZR{1362}$ and $g\geq 3$. It
would be interesting to determine more values of $\fasd(\cdot,\cdot)$. In
this section, we conclude our paper by listing some of the problems we are
interested in but unable to solve. 

In this paper we determined $\fasd(\Delta,3)$ for all $\Delta\geq 3$ except
$\Delta = 5$. We would like to conjecture the following, proving which
would generalize $\fasd(4,3)=3$ and $\fas(5,3)=\frac{1}{3}$ proved in
\cite{GLYZ}.

\begin{conj}\label{conj:1}
	$\fasd(5,3)=3$.
\end{conj}

Determining the exact values of the other terms of $\fasd(\cdot, \cdot)$
would also be very interesting. As the first step towards determining more
values, we would like to pose the following problem for which we have shown
that $\fasd(3,6)\in \{5,6\}$ and $\fasd(4,4)\in \{3,4\}$.

\begin{problem}\label{prob:336}
	What is $\fasd(3,6)$ and $\fasd(4,4)$? 
\end{problem}

Note that $\fasd(\Delta,g)\leq g$ is a trivial bound for all $\Delta\geq 2$
and $g\geq 3$.} \GGn{It would be interesting to know when the trivial bound
    is not tight.} \YZn{Therefore, we have the following problem as it can
        be seen from \MAYZR{Table}~\ref{table:Results} that $\fasd(\Delta, 3)=2<3$
        for all $\Delta\geq 6$ and to determine such $g$ for $\Delta=5$ is
    equivalent to solving our Conjecture \ref{conj:1}}.

\begin{problem}
    \label{prob:fasdg}
    For $\Delta=3$ or $4$, what is the \MAYZR{least} integer $g\geq 3$ such
    that $\fasd(\Delta,g)\MAYZR{\ <\ }g$?
\end{problem}

It can be seen from Fig. \ref{table:Results} that such $g$ belongs to
\MAYZR{$\{6,7,8\}$} if $\Delta=3$ and belongs to $\{4,5,6\}$ if $\Delta=4$.

\MAYZR{We would also like to pose the following analogue of Problem~\ref{prob:fasdg} for $\fas(\Delta,g).$

\begin{problem}
    \label{prob:fasg}
    For $\Delta \in \{3,4,5\}$, what is the least integer $g\geq 3$ such that
    $\fas(\Delta,g) > 1/g$?
\end{problem}

Note that for $\Delta=6$, Proposition~$11$ of \cite{GLYZ} shows that
$\fas(\Delta,g) < 1/g$ already for $g=3$. We briefly discuss the possible
values of the $g$ in Problem~\ref{prob:fasg} for each $\Delta \in
\{3,4,5\}$. It is straightforward to verify that the digraph $D_8$
illustrated in Figure~\ref{D-digraph} has $\fas(D_8) = 2$ and $a(D_8) =
15$. Thus, $\fas(3,8) \geq 2/15 > 1/8$. As $\fas(3,g) = 1/g$ for all $g \in
\{3,4,5,6\}$ (proven in Section~\ref{sec:d=3}), we have $g \in \{7,8\}$ for
$\Delta=3$. As $\fas(4,3) = \fas(5,3) = 1/3$ (Theorem $2$ of \cite{GLYZ})
and $\fas(\Delta+1,g) \leq \fas(\Delta,g)$, we have $4 \leq g \leq 8$ for
$\Delta=4$ and $5$. 
}

We have shown that $\fasd(\Delta,g)\leq \FIX{94}$ for all $\Delta\geq 3$ and $g\geq 3$. It would be great to know an accurate upper bound for every fixed $\Delta$ and therefore we pose the following problem.

\begin{problem}
	Given a fixed $\Delta\geq 3$, what is $\lim\limits_{g\to \infty}\fasd(\Delta,g)$? \MAYZR{In particular, what is $\lim\limits_{g\to \infty}\fasd(3,g)$?}
\end{problem}

\MAYZR{As we have shown that $\fas(3,g)=\frac{1}{g}$ for $g\in \{3,4,5,6\}$. We would like to pose the following similar question. 

\begin{problem}
	Given a fixed $\Delta\geq 3$, what is $\lim\limits_{g\to \infty}\fas(\Delta,g)$? In particular, what is $\lim\limits_{g\to \infty}\fas(3,g)$?
\end{problem}}

We known that for every $g\geq 3$, $\fasd(\Delta,g)=2$ for all $\Delta\geq \MAYZR{1362}$. It would interesting to determine the smallest $\Delta$ such that $\fasd(\Delta,g)=2$ for arbitrarily large $g$.

\begin{problem}
	What is the smallest $\Delta$ such that $\fasd(\Delta,g)=2$ for arbitrarily large $g$?
\end{problem}

\paragraph{Acknowledgement.} We are thankful to Jiangdong Ai and  Xiangzhou Liu for useful suggestions on the paper.

\end{document}